\documentclass[12pt,a4paper]{amsart}
\usepackage{amssymb,amsmath}
\usepackage{cancel}
\usepackage{palatino}
\usepackage{longtable}
\usepackage[hmargin=3cm,vmargin=3cm]{geometry}
\usepackage{color}
\numberwithin{equation}{section}
\newtheorem{theorem}{Theorem}[section]     
\newtheorem{definition}[theorem]{Definition}
\newtheorem{defi}[theorem]{Definition}
\newtheorem{proposition}[theorem]{Proposition}
\newtheorem{lemma}[theorem]{Lemma}
\newtheorem{rmk}[theorem]{Remark}
\newtheorem{corollary}[theorem]{Corollary}

\newtheorem{remark}[theorem]{Remark}

\newcommand{\cL}{\mathcal{L}}

\def\d{\partial}

\def\f{\frac}

\def\proof{\noindent\hspace{2em}{\itshape Proof: }}
\def\QEDclosed{\mbox{\rule[0pt]{1.3ex}{1.3ex}}} 

\def\QED{\QEDclosed} 
\def\endproof{\hspace*{\fill}~\QED\par\endtrivlist\unskip}
\newcommand{\eqa}{\begin{eqnarray}}
\newcommand{\eeqa}{\end{eqnarray}}
\newcommand{\beq}{\begin{equation}}
\newcommand{\eeq}{\end{equation}}

\numberwithin{equation}{section}

\begin{document}
\title{A Dubrovin-Frobenius manifold structure of NLS type on the orbit space of $B_n$}

\author{Alessandro Arsie}
\address{A.~Arsie:\newline Department of Mathematics and Statistics, The University of Toledo,\newline 2801W. Bancroft St., 43606 Toledo, OH, USA}
\email{alessandro.arsie@utoledo.edu}

\author{Paolo Lorenzoni}
\address{P.~Lorenzoni:\newline Dipartimento di Matematica e Applicazioni, Universit\`a di Milano-Bicocca, \newline
Via Roberto Cozzi 53, I-20125 Milano, Italy and INFN sezione di Milano-Bicocca}
\email{paolo.lorenzoni@unimib.it}

\author{Igor Mencattini}
\address{I.~Mencattini:\newline Instituto de Ci\^encias Matem\'aticas e de Computa\c c\~ao, Universidade de S\~ao Paulo,\newline
S\~ao Carlos, SP, Brazil}
\email{igorre@icmc.usp.br}

\author{Guglielmo Moroni}
\address{G.~Moroni:\newline Dipartimento di Matematica e Applicazioni, Universit\`a di Milano-Bicocca, \newline
Via Roberto Cozzi 53, I-20125 Milano, Italy and INFN sezione di Milano-Bicocca}
\email{g.moroni9@campus.unimib.it}


\maketitle

\begin{abstract}
Generalizing a construction presented in \cite{ALcomplex}, we show that the orbit space of $B_2$ less the image of coordinate lines under the quotient map is equipped with two Dubrovin-Frobenius manifold structures which are related respectively to the defocusing and the focusing nonlinear Schr\"odinger (NLS) equations. Motivated by this example, we study the case of $B_n$ and we show that the defocusing case can be generalized to arbitrary $n$ leading to a Dubrovin-Frobenius manifold structure on the orbit space of the group. The construction is based on the existence of a non-degenerate and non-constant invariant bilinear form that plays the role of the Euclidean metric in the  Dubrovin-Saito standard setting. Up to $n=4$ the prepotentials we get coincide with those
 associated with constrained KP equations discussed in \cite{LZZ}.
\end{abstract}

\tableofcontents

\section{Introduction}
According to a classical theorem due to Chevalley, given a finite group generated by (pseudo-)reflections the invariant functions ring of the orbit space of the group is  a polynomial ring generated by
 a set of invariant polynomials,  called basic invariant polynomials.

In general, the basic invariants are not uniquely defined, while their
degrees depend only on the choice of the  group.

 In the case of Coxeter groups, a procedure to select uniquely a set of basic invariant polynomials was proposed
 by Saito in \cite{Sa} and it is based on the notion of flat structure on the orbit space.  
 
 An explicit construction of polynomial basic invariants was implemented
 by Saito, Yano and Sekiguchi in \cite{SYS} through a case by case analysis (with the exception of the group $E_7$ and $E_8$).

In 1993 Dubrovin interpreted Saito's construction in terms of bihamiltonian geometry and Dubrovin-Frobenius manifolds \cite{DCoxeter}. He showed that starting
 from the Euclidean metric (defined on the Euclidean space where the group acts) and from the flat structure on the orbit space, 
 it is possible to define a flat pencil of metrics. This notion had been previously introduced by Dubrovin himself in the study of a special class
 of bihamiltonian structures related to Dubrovin-Frobenius manifolds \cite{du97}. 
Under suitable additional assumptions (exactness, homogeneity and Egorov property)
 flat pencil of metrics are in one-to-one correspondence with Dubrovin-Frobenius
 manifolds. Using this correspondence, he defined a polynomial Dubrovin-Frobenius manifold structure on the orbit space of Coxeter groups. The polynomiality of the prepotential for any Coxeter group was conjectured by Dubrovin and proved by Hertling
 in \cite{He}. It was observed in \cite{DaZ} that in the case of groups $B_n$ and $D_n$ there are different possible choices of the unit vector field leading to different Dubrovin-Frobenius manifold structures.

In 2004, in the paper \cite{Dalmost}, Dubrovin introduced the notion of almost duality and showed that in the case of a Coxeter group the almost dual structure coincides
 with a universal structure introduced by Veselov in \cite{Ve} (Veselov's $\vee$-system). In the same paper he found a generalization of Saito's construction
 for Shephard groups (symmetry groups of regular complex polytopes \cite{Sh1}). The role of the Euclidean metric in this case is played by a flat metric defined by the
 Hessian of the lowest degree basic invariant. Flatness of this metric relies on a previous result of Orlik and Solomon \cite{OS}.
 
 It turned out that the Dubrovin-Frobenius structure obtained in this way on the orbit space of a Shephard group is isomorphic to the Dubrovin-Frobenius structure
 defined on the orbit space of the associated Coxeter group.
 
In 2015 Kato, Mano and Sekiguchi proposed a further generalization of Dubrovin-Saito construction in the case of well-generated complex reflection groups \cite{KMS2}. 
 The outcome of their construction is not a Dubrovin-Frobenius manifold but a flat F-manifold \cite{manin} or, using the language of meromorphic connections, a Saito structure without metric \cite{S}.  

In 2017 two of the authors of the present paper proposed an alternative construction of (bi)-flat F-manifolds on the orbit space of complex reflection groups \cite{ALcomplex}. The starting point of \cite{ALcomplex} is a ''dual flat structure" defined by a family of flat connections of Dunkl-Kohno type associated with a complex reflection group \cite{DO,Ko,Looijenga}. This family of connections depends on the choice of an invariant function on the set of reflecting hyperplanes: for each hyperplane one has to choose a "weight" and the weights assigned
 to different hyperplanes must coincide if the hyperplanes belong to the same orbit under the action of the relevant group. 
 
  A standard choice consists in assigning to each hyperplane the order of the corresponding
 reflection. In all the examples considered in \cite{ALcomplex} this choice corresponds
 to Kato-Mano-Sekiguchi flat F-manifold structure. Other admissible choices lead to
 different structures and conjecturally the orbit space of a well-generated complex reflection is equipped with a $(N-1)$-parameter family of flat F-manifold structures,
 where $N$ is the number of orbits for the action of the group on the set of reflecting
 hyperplanes, see \cite{ALPreviato}. This conjecture has been verified for Weyl groups of rank $2$, $3$ and  $4$, for the dihedral groups $I_2(m)$, for any of the exceptional well-generated complex reflection groups of rank $2$ and $3$ and for any of the groups $G(m,1,2)$
 and $G(m,1,3)$. In \cite{ALcomplex} it was also pointed out that in the case of Shephard
 groups, in general, the Kato-Mano-Sekiguchi construction does not reduce to Dubrovin's construction on the orbit space of these groups.

An alternative proof of the existence of the "standard" Kato-Mano-Sekiguchi structure was obtained starting from a dual structure (equivalent to the flat structure considered in \cite{ALcomplex})
 by Konishi, Minabe and Shiraishi in \cite{KMSh}.

In the present paper, combining the Dubrovin-Saito approach with the approach pursued in \cite{ALcomplex}, we present a further generalization of the Dubrovin-Saito procedure for the series $B_n$. In the first part of the paper, exploiting the flexibility of the second approach
 we study flat F-manifold structures obtained from a dual flat structure of the form outlined above in the case of $B_2,B_3$ and $B_4$. In all these cases besides the one-parameter family obtained in \cite{ALcomplex,ALPreviato} there are additional Dubrovin-Frobenius structures associated with a suitable choice of the weights in the definition of the dual connection and of the dual product. The corresponding solutions of WDVV equations are no longer polynomial due to appearance of a logarithmic term. 
For $n=2$ it coincides with the Frobenius manifold structure
associated with focusing and defocusing NLS equation depending on the choice of the weights: assigning weight zero to the coordinate axes and a non vanishing weight to the remaining mirrors one gets the defocusing case, while the opposite choice leads to the  focusing case. The first choice survives also in the case $n=3$ and in the case $n=4$
 leading to similar solutions of WDVV equations. These solutions appear in literature (for arbitrary $n$) in connection with constrained KP equation, see \cite{LZZ}.
As a byproduct of these computations, we get a bilinear form invariant with respect to the action of $B_n$.
 In order to prove the existence of a Dubrovin-Frobenius structure for any $n$, in the second part of the paper, we apply the Dubrovin-Saito procedure to
 the invariant bilinear form obtained in the first part. The main difficulty encountered in the present case, if compared with the standard one, is due to the fact the flat pencil obtained  
 applying the first part of the procedure is not regular and, as a consequence, it is not possible to define all the structure constants of the product in terms of
 the Christoffel symbols of the intersection form. This is also the reason of the presence of a logarithmic term in the Dubrovin-Frobenius prepotential.
\newline
\newline
\noindent{\bf Acknowledgements}.  We thank Giordano Cotti for useful comments. P.~L. is supported by funds of H2020-MSCA-RISE-2017 Project No. 778010 IPaDEGAN.
 Authors are also thankful to GNFM - INdAM for supporting activities that contributed to the research reported in this paper.

\section{Bi-flat $F$-manifolds and Frobenius manifolds}\label{sec:biflat}
\begin{defi}\cite{manin}\label{defflatFman}
A {\rm flat $F$-manifold} is a quadruple $(M,\circ,\nabla,e)$ where $M$ is a manifold, 
 $\circ : \mathcal{X}(M) \times \mathcal{X}(M) \rightarrow \mathcal{X}(M)$ is a product ($\mathcal{X}(M)$ is the $C^{\infty}$-module of local vector fields),
$\nabla$ is a connection on the tangent bundle $TM$
and $e$ is a distinguished vector field, satisfying the following axioms:
\begin{enumerate}
\item for every $\lambda\in\mathbb R$, $\nabla_{(\lambda)}:=\nabla+\lambda \circ$ 
is a flat and torsionless connection.
\item $e$ is the unit of the product $\circ$, i.e. for any $X\in \mathcal{X}(M)$, one has $e\circ X=X\circ e=X$.
\item $e$ is flat: $\nabla e=0$.
\end{enumerate}
\end{defi}
Manifolds  equipped with a product $\circ$, a connection $\nabla$ and a vector field $e$ satisfying conditions $(1)$ and $(2)$ above will be called {\it almost flat $F$-manifolds}.

In local coordinates $(u_1, \dots u_n)$, denoting with $c^i_{jk}$ the structure constants of the product $\circ$ and with $\Gamma^i_{jk}$ the Christoffel symbols of the connection $\nabla$, Condition $1$ in Definition \ref{defflatFman} reads 
\beq\label{defortor}
T^{(\lambda)k}_{ij}=T^k_{ij}+\lambda(c^k_{ij}-c^k_{ji})=0,
\eeq
and
\beq\label{deforcurv}
R^{(\lambda)k}_{ijl}=R^k_{ijl}+\lambda(\nabla_i c^k_{jl}-\nabla_j c^k_{il})+\lambda^2(c^k_{im}c^m_{jl}-c^k_{jm}c^m_{il})=0,
\eeq
where $T^{(\lambda)k}_{ij}$ and $R^{(\lambda)k}_{ijl}$ are the torsion and the curvature tensor of the connection $\nabla_{(\lambda)}$, while $T^k_{ij}$ and $R^k_{ijl}$ are the torsion and the curvature tensor of the connection $\nabla$. The identity principle of polynomials applied to \eqref{defortor} and \eqref{deforcurv} yields the following consequences: 
\begin{enumerate}
\item the connection $\nabla$ is torsionless,
\item the product  $\circ$ is commutative,
\item the connection $\nabla$ is flat,
\item the tensor field $\nabla_l c^k_{ij}$ is symmetric in the lower indices,
\item the product $\circ$ is associative.
\end{enumerate}
From the above conditions it follows that in flat coordinates the structure constants of 
 the product can be written as second order partial derivatives of a vector field 
\beq\label{eqvectpot}c^i_{jk}=\d_j\d_k F^i,\eeq
satisfying a non-trivial system of PDEs called \emph{generalized WDVV equations} or \emph{oriented associativity equations}:
\beq\label{asseqvectpot}
\d_j\d_l F^i\d_k\d_mF^l=\d_k\d_l F^i\d_k\d_mF^l.
\eeq
Dubrovin-Frobenius manifolds are  flat F-manifolds equipped with a homogeneous invariant
 pseudo-Riemannian metric $\eta$ compatible with the connection $\nabla$. More precisely
\begin{defi}\label{Frobeniusdef}
A Dubrovin-Frobenius manifold is a flat $F$-manifold  $(M, \circ, \nabla, e)$ equipped with a metric $\eta$ and a distinguished vector field $E$, called the Euler vector field, satisfying the following conditions
\beq\label{eq2}
\nabla\eta=0,
\eeq
\beq\label{eq3}
\eta(X\circ Y,Z)=\eta(X, Y\circ Z), \; \forall X, Y, Z\in \mathcal{X}(M),
\eeq
$$[e,E]=e,\quad {\rm Lie}_E \circ=\circ,$$
and
$${\rm Lie}_E\eta=(2-d)\eta,$$
where $d$ is a constant called the charge of the Frobenius manifold. 
The latter requirement means that $E$ acts as a conformal Killing vector field of the metric $\eta$.
\end{defi}
In flat coordinates, the existence of an invariant metric implies $F^i=\eta^{il}\d_l F$ 
 for a scalar function $F$ called the prepotential of the Dubrovin-Frobenius manifold. Using this
fact, it is immediate to see that the associativity equations \eqref{asseqvectpot} become the usual WDVV associativity equations:
\beq\label{WDVVeq}
\d_j\d_h\d_i F\eta^{il}\d_l\d_k\d_m F=\d_j\d_k \d_iF\eta^{il}\d_l\d_h\d_m F.
\eeq

It is worth noticing that every Dubrovin-Frobenius manifold comes together with \emph{an almost dual}, i.e. a \emph{second} almost flat $F$-manifold structure. More precisely
\begin{theorem}\cite{Dalmost}\label{thalmostduality}
Given a Dubrovin-Frobenius manifold $(M, \circ, e, E, \eta, \nabla)$, consider the open set $U$ where the endomorphism of the tangent bundle $E\circ$ is invertible and consider the corresponding intersection form, i.e. the contravariant metric $g:=(E\circ)\,\eta^{-1}$. Then on $U$, the data given by 
\begin{enumerate}
\item the Levi-Civita connection $\tilde\nabla$ of $g$, 
\item the Euler vector field $E$ and 
\item a  {\rm dual product} defined as $X*Y=(E\circ)^{-1}\,X\circ Y,\quad\forall X,Y\in \mathcal{X}(U)$,
\end{enumerate}
define an almost flat $F$-manifold with unit $E$ and invariant metric $g^{-1}$. 
\end{theorem}

 Replacing $\tilde\nabla$ with $\nabla^*:=\tilde\nabla+\bar\lambda*$ (for a suitable value of $\bar\lambda$) one obtains a flat connection $\nabla^*$ satisfying  $\nabla^* E= 0$. In this way, for any given Dubrovin-Frobenius manifold $(M,\eta,\circ,e,E,\nabla)$, there are two flat structures:
\begin{itemize}
\item the ``natural" flat structure $(\nabla,\circ,e)$, 
\item the ``dual" flat structure $(\nabla^*,*,E)$.
\end{itemize}
It turns out that these two structures are related by the following condition:
\beq\label{eq5}
(d_{\nabla}-d_{\nabla^*})(X\circ)=0, \; \forall X\in \mathcal{X}(U),
\eeq
where $d_{\nabla}$ is the exterior covariant derivative. 
\begin{defi}\cite{ALbiflat}\label{biflatdef}
A {\rm bi-flat $F$-manifold} $M$ is a manifold equipped with two different flat $F$-structures $(\nabla,\circ,e)$ and  $(\nabla^{*},*,E)$
related by the following conditions
\begin{enumerate}
\item $E$ is an Euler vector field.
\item $*$ is the dual product defined by $E$.
\item  $\nabla$ and $\nabla^*$ satisfy condition \eqref{eq5}. 
\end{enumerate}
\end{defi}

\section{Bi-flat F-manifolds and complex reflection groups}\label{biflatcomplex}
A complex (pseudo)-reflection is a unitary transformation of $\mathbb{C}^n$ of finite period that leaves invariant a hyperplane.
 A complex reflection group is a finite group generated by (pseudo)-reflections.
  Irreducible finite complex reflection groups were classified by Shephard and Todd in \cite{ST} and consist of an infinite family depending on $3$ positive integers and $34$ exceptional cases.
 Well-generated irreducible complex reflection groups are irreducible complex reflection groups of rank $n$ generated by $n$ (pseudo)-reflections.

\subsection{Flat structures associated with Coxeter groups}
A Coxeter group is automatically well-generated. For Coxeter groups we have the following result.
\begin{theorem}[Dubrovin, \cite{DCoxeter}]\label{Dcoxeterth}
The orbit space of a finite Coxeter group is equipped with a Dubrovin-Frobenius manifold structure $(\eta,\circ,e,E)$ where
\begin{enumerate}
\item The invariant metric $\eta$ coincides with the bilinear form constructed in
 \cite{Sa,SYS}. The corresponding set of basic invariant are called {\it Saito flat coordinates}. 
\item In the Saito flat coordinates 
$$e=\f{\d}{\d u_n},\,\,E=\sum_{i=1}^n\left(\f{d_i}{d_n}\right)u_i\f{\d}{\d u_i}.$$
\end{enumerate}
where $d_i$ are the degrees of the invariant polynomials $u_i$
 and $2=d_1<d_2\le d_3\le\dots\le d_{n-1}<d_n$ ($d_n$ is the Coxeter number). 
\end{theorem}

Dubrovin-Saito construction relies on the existence of a flat pencil of metrics associated with any Coxeter group. Let us illustrate this construction in a simple example.

\subsection{Dubrovin-Saito construction for $B_2$}
In this case, the basic invariants  have degree $d_1=2$ and $d_2=4$. Up to a constant factor they
 have the form
\[u_1=\f{1}{8}(p_1^2+p_2^2),\qquad u_2=p_1^2p_2^2+cu_1^2\]
where $c$ is an arbitrary constant. The Euclidean cometric has the standard constant form in the coordinates $(p_1 , p_2)$. Rewriting the Euclidean cometric in the coordinates $(u_1,u_2)$ we get
\[g=\begin{pmatrix} \f{1}{2}u_1 & u_2\\ u_2 & -2c(c+16)u_1^3+4(c+8)u_1u_2
\end{pmatrix}.\]
According to Saito's general result there is a unique choice of $c$ such that the cometric
 $\eta=\mathcal{L}_{\frac{\partial}{\partial u_2}}g$ is non-degenerate and constant. Indeed the cometric
\[\eta=\begin{pmatrix} 0 & 1\\ 1 & 4(c+8)u_1
\end{pmatrix}\]
is constant only if $c=-8$.

According to Dubrovin's general result for such a choice of $c$ the pencil 
 $g-\lambda\eta$ is a flat pencil of contravariant metric satisfying the following additional properties
\begin{itemize}
\item \emph{Exactness}: there exists a vector field $e$ such that
\[
\mathcal{L}_e g=\eta,\qquad \mathcal{L}_e \eta=0.
\]
\item \emph{Homogeneity}: 
\[
\mathcal{L}_E g=(d-1)g,
\]
where $E^i:=g^{il}\eta_{lj}e^j$.
\item \emph{Egorov property}: locally there exists a function $\tau$ such that
\[
e^i=\eta^{is}\d_s\tau,\qquad E^i=g^{is}\d_s\tau.
\]
\end{itemize}
Indeed it is immediate to check that $e^i=\delta^i_2$, $E^i=\f{d_i}{4}u_i$
 and $\tau=u_1$. The corresponding solution
 of WDVV equation is obtained solving the system \cite{DLectures}
\[\f{d_i+d_j-2}{h}\eta^{il}\eta^{jk}\d_l\d_m F=g^{ij}.\]
Up to inessential linear terms the solution is
\[F=\f{1}{2}u_1u_2^2+\f{64}{15}u_1^5.\]

\subsection{Almost dual structure and Veselov's $\vee$-system}
In the case of Coxeter groups the almost dual structure has a special form, whose structure is independent
 of the choice of the group. It is defined by the data
$$\left(\nabla^*,\quad*=\f{1}{N}\sum_{H\in \mathcal{H}}\frac{d\alpha_H}{\alpha_H}\otimes\pi_H,\quad E=\sum p_k\f{\partial}{\partial p_k}\right)$$
where 
\begin{itemize}
\item $\nabla^*$ is the Levi-Civita connection of the Euclidean metric,
\item $\mathcal{H}$ is  the collection of the  reflecting hyperplanes  $H$, 
\item $\alpha_H$ is a linear form defining the reflecting hyperplane $H$, 
\item $\pi_H$ is the orthogonal projection onto the orthogonal complement of $H$, 
\item $N$ is a normalization factor.
\end{itemize} 
Products of this form appear in the work of Veselov on $\vee$-systems \cite{Ve} (see \cite{ALjmp,FV} for an interpretation of Veselov's conditions in terms
 of flatness of a Dunkl-Kohno type connection).

\subsection{Flat structures associated with complex reflection groups}
Dubrovin-Saito flat structure and Veselov's dual structure can be generalized to
 complex reflection groups.  
\begin{theorem}\cite{KMS2}\label{ALcomplexthm1}
The orbit space of a well-generated complex reflection group is equipped with a flat $F$-structure $(\nabla,\circ,e,E)$ with {\rm linear} Euler vector field where
\begin{enumerate}
\item The flat coordinates for $\nabla$ are basic invariants $(u_1,\dots,u_n)$ of the group (generalized Saito coordinates).
\item In the Saito flat coordinates 
$$e=\f{\d}{\d u_n},\,\,E=\sum_{i=1}^n\left(\f{d_i}{d_n}\right)u_i\f{\d}{\d u_i}.$$
\end{enumerate}
\end{theorem}

\begin{rmk}\label{rwkms}
The linearity of $E$ (i.e. the condition $\nabla\nabla E=0$) turns out to be equivalent to the existence of a second compatible flat structure. This was proved in the semisimple case in \cite{ALcomplex} and later in the non-semisimple case (under some regularity assumptions) in \cite{KMSh}.
\end{rmk}

The dual flat structures are described by the following theorem 
\begin{theorem}\label{ALcomplexthm2}
Let $G$ be an irreducible complex reflection group acting on $\mathbb{C}^n$. Then the data
$$\left(\nabla^*
=\nabla^0-\sum_{H\in \mathcal{H}}\frac{d\alpha_H}{\alpha_H}\otimes\tau_H\pi_H,\, *=\sum_{H\in \mathcal{H}}\frac{d\alpha_H}{\alpha_H}\otimes\sigma_H\pi_H,\, E=\sum p_k\f{\partial}{\partial p_k}\right)$$
where 
\begin{itemize}
\item $\mathcal{H}$ is  the collection of the  reflecting hyperplanes  $H$, 
\item $\alpha_H$ is a linear form defining the reflecting hyperplane $H$, 
\item $\pi_H$ is the unitary projection onto the unitary complement of $H$,
\item the collections of weights $\sigma_H$ and $\tau_H$ are $G$-invariant and satisfy
\beq\label{sums} 
\sum_{H\in \mathcal{H}}\sigma_H\pi_H=\sum_{H\in \mathcal{H}}\tau_H\pi_H=Id.
\eeq
\item $\nabla^0$ is the standard flat connection on $\mathbb{C}^n$,
\end{itemize} 
define a flat $F$-structure on the orbit space of the action of $G$ on $\mathbb{C}^n$.  
\end{theorem}

In the case of well generated complex reflection groups of rank 2,3,4 it was proved in 
 \cite{ALcomplex,ALPreviato} that, for a suitable choice of the weights $\sigma_H$  and $\tau_H$ and of the basic invariants, there exists a bi-flat $F$-manifold structure whose natural
 structure has the form described in Theorem \ref{ALcomplexthm1} and whose dual
 structure has the form described in Theorem \ref{ALcomplexthm2}. In all the examples
 choosing $\sigma_H$ and $\tau_H$, proportional to the order of the corresponding pseudo-reflection the natural structure coincides with the flat structures obtained in \cite{KMS2}. In general the choice of the weights $\tau_H$ is not unique as we are going to illustrate in the case of $B_2$.
  
\subsection{A simple example: $B_2$}
\subsubsection{Step 1. The dual product $*$}
We start from the product
$$*=\sum_{H\in \mathcal{H}}\frac{d\alpha_H}{\alpha_H}\otimes\sigma_H\pi_H$$
where
\begin{eqnarray*}
&&\alpha_1=p_1,\qquad \alpha_2=p_2,\qquad \alpha_3=p_1-p_2\qquad\alpha_4=p_1+p_2
\end{eqnarray*}
Let us call Orbit 1 the orbit containing the straight lines $\alpha_1=0$ and $\alpha_2=0$  and Orbit 2 the orbit containing the straight lines $\alpha_3=0$ and $\alpha_4=0$. 
According to the general rule the weights must be the same for lines in the same orbit:  
\begin{eqnarray*}
&&\sigma_{1}=\sigma_{2}=\f{x}{x+y},\qquad \sigma_{3}=\sigma_{4}=\f{y}{x+y}
\end{eqnarray*}
We get
\begin{eqnarray*}
&&c^{*1}_{11}=\f{(x+y)p_1^2-xp_2^2}{(x+y)p_1(p_1^2-p_2^2)},\quad
c^{*2}_{11}=\f{-yp_2}{(x+y)(p_1^2-p_2^2)}=c^{*1}_{21}=c^{*1}_{21}\\
&&c^{*2}_{12}=\f{yp_1}{(x+y)(p_1^2-p_2^2)}=c^{*2}_{21}=c^{*1}_{22},\quad
c^{*2}_{22}=\f{xp_1^2-(x+y)p_2^2}{(x+y)p_2(p_1^2-p_2^2)}.
\end{eqnarray*}

\subsubsection{Step 2. The connection $\nabla$} 
We assume that the basic invariants are flat coordinates of $\nabla$. For $B_2$ up to a constant factor they depend on a single parameter $c$:
$$u_1=p_1^2+p_2^2,\qquad u_2=p_1^4+p_2^4+cu_1^2.$$ 
Writing the connection $\nabla$ in the coordinates $p_1,p_2$ we get
\begin{eqnarray*}
&&\Gamma^1_{11}=-\f{(4c-1)p_1^2+p_2^2}{p_1(p_1^2-p_2^2)},\quad\Gamma^2_{11}=\f{4cp_1^2}{p_2(p_1^2-p_2^2)},
\quad\Gamma^1_{12}=-\f{2(2c+1)p_2}{(p_1^2-p_2^2)}=\Gamma^1_{21}\\
&&\Gamma^2_{12}=\f{2(2c+1)p_1}{(p_1^2-p_2^2)}=\Gamma^2_{21},\quad
\Gamma^1_{22}=-\f{4cp_2^2}{p_1(p_1^2-p_2^2)},\quad
\Gamma^2_{22}=\f{(4c-1)p_2^2+p_1^2}{p_2(p_1^2-p_2^2)}.
\end{eqnarray*}

\subsubsection{Step 3. The unit vector field $e$} 
We assume that in the basic invariants $e=\frac{\partial}{\partial u_2}$.
\subsubsection{Step 4. The product $\circ$}
From $*$ and $e$ we can define $\circ$ in the usual way as
$$X\circ Y=(e*)^{-1}X*Y,\quad \forall X,Y.$$
We get
\begin{eqnarray*}
&&c^1_{11}=-\f{2xp_1^3}{(x+y)}+2p_1p_2^2,\quad
c^2_{11}=\f{2yp_1^2p_2}{x+y}=c^1_{12}=c^1_{21}\\
&&c^2_{12}=\f{2yp_2^2p_1}{x+y}=c^2_{21}=c^1_{22},\quad
c^2_{22}=2p_2p_1^2-\f{2xp_2^3}{x+y}.
\end{eqnarray*}
\subsubsection{Step 5. The constraint on the weights}
Imposing the compatibility between $\nabla$ and $\circ$:
$$\nabla_k c^i_{jl}=\nabla_jc^i_{lk}$$
we get the constraint $x=y$, that is $\sigma_1=\sigma_2=\sigma_3=\sigma_4=\f{1}{2}$. 
\subsubsection
{Step 6. The dual connection $\nabla^*$} 
Imposing the condition $\nabla^* E=0$ and the condition \eqref{eq5} we obtain
\begin{eqnarray*}
&&b^1_{11}=-\f{(4c+1)p_2^2-p_1^2}{p_1(p_1^2-p_2^2)},\quad b^2_{11}=-\f{4cp_2}{p_1^2-p_2^2},
\quad b^1_{12}=-\f{-4cp_2}{p_1^2-p_2^2}=b^1_{21}\\
&&b^2_{12}=\f{4cp_1}{p_1^2-p_2^2}=b^2_{21},\quad
b^1_{22}=-\f{4cp_1}{p_1^2-p_2^2},\quad
b^2_{22}=-\f{(4c+1)p_2^2-p_2^2}{p_2(p_1^2-p_2^2)}.
\end{eqnarray*}
In particular for $c=-\f{1}{8}$ we have $b^i_{jk}=-c^{*i}_{jk}$. 
\subsubsection{Step 7. The vector potential}
The above data and the Euler vector field $E=\sum _n p_n\f{\partial}{\partial p_n}$ define a a bi-flat $F$-manifold structure $(\nabla,\circ,e,\nabla^*,*,E)$ for {\em  any choice} of  $c$. 
 Solving the system
 $$c^i_{jk}=\d_j\d_k F^i_{B_2},$$
 we get the vector potential
\begin{equation}\label{bfB2}
F_{B_2}^1= u_1u_2-\f{1}{12}u_1^3(8c+1),\qquad F_{B_2}^2= 
-\f{c}{12}(4c+1)u_1^4+\f{1}{2}u_2^2.
\end{equation}
For $c=-\f{1}{8}$ the vector potential comes from a Dubrovin-Frobenius prepotential.

Summarizing, assuming $e=\frac{\partial}{\partial u_2}$ the choice of the
 weights $\sigma_H$ is unique (they coincide up to a normalization factor with the order
 of the corresponding reflection) while the choice of the weights $\tau_H$ depends on
 a parameter $c$. In \cite{ALcomplex,ALPreviato} it was conjectured that this additional freedom appears every time that  all the mirrors do not belong to the same orbit.
 
\section{A modified construction}
\subsection{The case of $B_2$}
In flat coordinates the components of the unit vector field should be constant. Following
 Dubrovin-Saito and Kato-Mano-Sekiguchi we have assumed that the flat coordinates
 are basic invariants and that $e=\frac{\partial}{\partial u_2}$, where $u_2$ is the highest degree invariant polynomial. The last assumption is very natural since (up to a constant factor) the vector field $\frac{\partial}{\partial u_2}$ is not affected by a change in the choice of the basic invariants. In this Section, restricting ourself to the case of $B_2$, we will study what happens if we remove this hypothesis.
\newline
\newline
\noindent
{\bf Modified Step 3: we assume that in the basic invariants $e$ is constant.}
\newline
\newline
Defining $\circ$ in the usual way and imposing the condition
$$\nabla_k c^i_{jl}=\nabla_jc^i_{lk},$$
after some computations (performed with the help of Maple) we get the following solutions
\begin{enumerate} 
\item $y=x$, $e^1=0$,
\item $c=0$, $x=0$ and $e^2=0$,
\item $c=-\f{1}{4}$, $y=0$ and $e^2=0$.
\end{enumerate}
The first solution corresponds to the one-parameter family of bi-flat $F$-manifold structures related to the vector potentials \eqref{bfB2}. Following the same steps outlined above, the second and the third solution lead to the following solutions of WDVV equations 
\begin{eqnarray*}
F&=& \f{1}{2}u_1^2u_2\pm\f{1}{2}u_2^2\left(\ln{u_2}-\f{3}{2}\right).\\
\end{eqnarray*}
These are the prepotentials of the Dubrovin-Frobenius manifolds associated with
 defocusing/focusing NLS equation. Indeed, let us consider the chain of commuting
 flows of the principal hierarchy (see for instance \cite{DLectures}), obtained starting from
\[u^i_{t_0}=u^i_x,\qquad i=1,2.\]
These flows have the form
\[u^i_{t_{(\alpha)}}=c^i_{jk}X^j_{(\alpha)}u^k_x=\eta^{il}\d_l\d_j\d_k FX^j_{(\alpha)}u^k_x,\qquad i=1,2,\,\alpha=0,1,2,\dots\]
where $X^j_{(0)}=e^j=\delta^j_1$ and the vector fields $X_{(\alpha)}$ are obtained
 solving the recursion relations
\[\d_jX^i_{(\alpha)}=c^i_{jk}X^k_{(\alpha-1)}.\]
For instance (independently of the choice of the sign in $F$) we obtain
\[X^1_{(1)}=u^1,\qquad X^2_{(1)}=u^2.\]
Taking into account that the non-zero  structure constants are
\[c^1_{22}=\pm\f{1}{u^2},\qquad c^1_{11}=c^2_{12}=c^2_{21}=1,\]
the corresponding evolutionary PDEs  are given by
\begin{eqnarray*}
u^1_{t_{(1)}}&=&c^1_{jk}X^j_{(1)}u^k_x=c^1_{11}X^1_{(1)}u^1_x+
 c^1_{22}X^2_{(1)}u^2_x=u^1u^1_x\pm u^2_x,\\
u^2_{t_{(1)}}&=&c^2_{jk}X^j_{(1)}u^k_x=c^2_{12}X^1_{(1)}u^2_x+
 c^2_{21}X^2_{(1)}u^1_x=(u^1u^2)_x.
\end{eqnarray*}
They coincide with the dispersionless limit of the evolutionary system of PDEs associated with defocusing/focusing NLS equation (compare with Example 2.12 in \cite{Dams} where $u^1=-v$ and $u^2=u$).  

It is worth to mention that the genus expansion of the first Dubrovin-Frobenius manifold structure is related to  higher genera generalization of the Catalan numbers \cite{CLPS}. 

\bigskip

\begin{center}
\begin{picture}(140,140) 
\put(70,70){\textcolor{red}{\line(1,0){45}}} 
\put(70,70){\textcolor{red}{\line(0,1){45}}}
\put(70,70){\textcolor{red}{\line(0,-1){45}}}
\put(70,70){\textcolor{red}{\line(-1,0){45}}}
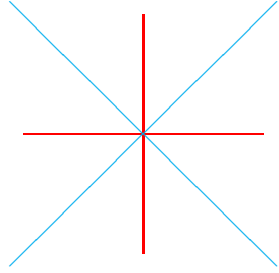
\put(-120,1){\scriptsize Reflecting hyperplanes for $B_2$: lines of the same colour corresponds to the same orbit type.}
\put(70,70){\textcolor{cyan}{\line(1,1){50}}}
\put(70,70){\textcolor{cyan}{\line(-1,-1){50}}}
\put(70,70){\textcolor{cyan}{\line(-1,1){50}}}
\put(70,70){\textcolor{cyan}{\line(1,-1){50}}}
\end{picture}
\end{center}

\subsection{The cases $B_3$ and $B_4$}
The previous computations becomes very cumbersome for $n>2$ and it seems very
 difficult to carry out all the steps without some additional assumptions.

Motivated by the previous example we investigate bi-flat $F$-manifold structures associated with following two
 choices of the weights $\{\sigma_H\}_{H\in \mathcal{H}}$:
\begin{enumerate}
\item $\sigma_H=0$ if $H$ is one of the (hyper)planes $p_i=0$ (otherwise $\sigma_H=1$). All these
 (hyper)planes belong to the same orbit (Orbit I).
\item 	$\sigma_H=0$ if $H$ is one of the (hyper)planes of Orbit II (otherwise $\sigma_H=1$). 
\end{enumerate}
It turns out that the first choice leads to a Dubrovin-Frobenius manifold with prepotentials
\begin{eqnarray*}
F_{B_3}&=&\f{1}{6}u_2^3+u_1u_2u_3+\f{1}{12}u_1^3u_3-\f{3}{2}u_3^2+u_3^2\ln{u_3},\\
F_{B_4}&=&\f{1}{108}u_1^4u_4+\f{1}{6}u_1^2u_2u_4-\f{1}{72}u_2^4+u_1u_3u_4+\f{1}{2}u_2^2u_4+\f{1}{2}u_2u_3^2-\f{9}{4}u_4^2+\f{3}{2}u_4^2\ln{u_4},
\end{eqnarray*}
while the second choice does not produce any bi-flat structure. 

 In order to prove the existence of a Dubrovin-Frobenius manifold structure for any $n$, associated with the first choice, we will use a different strategy. The key observation is that 
 in all the above examples ($n=2,3,4$) the intersection form has always the same
 expression
\[g_{B_2}=
\begin{bmatrix}
0&\frac{1}{p_1p_2}\\ 
\frac{1}{p_1p_2}&0
\end{bmatrix},\qquad 
g_{B_3}=
\begin{bmatrix}
0&\frac{1}{p_1p_2}&\frac{1}{p_1p_3}\\ 
\frac{1}{p_1p_2}&0&\frac{1}{p_2p_3}\\ 
\frac{1}{p_1p_3}&\frac{1}{p_2p_3}&0
\end{bmatrix},\qquad g_{B_4}=
\begin{bmatrix}
0&\frac{1}{p_2p_1}&\frac{1}{p_1p_3}&\frac{1}{p_1p_4}\\ 
\frac{1}{p_1p_2}&0&\frac{1}{p_2p_3}&\frac{1}{p_2p_4}\\ 
\frac{1}{p_1p_3}&\frac{1}{p_2p_3}&0&\frac{1}{p_3p_4}\\
\frac{1}{p_1p_4}&\frac{1}{p_2p_4}&\frac {1}{p_3p_4}&0
\end{bmatrix} 
\] 
In the next Section, starting from the intersection form defined by

\begin{equation}
g^{ij}(p)=\f{(1-\delta^{ij})}{p_ip_j},\qquad i,j=1,\dots,n\label{eq:c1}
\end{equation}
\noindent
we will prove the existence of a flat pencil of metrics which yields a Dubrovin-Frobenius structure
for any $n$. Our approach relies on a suitable generalization of Dubrovin-Saito
 construction. The proof of the existence of the Saito metric closely follows the ideas of the paper by Saito, Yano and Sekiguchi \cite{SYS},
 while the reconstruction of the Dubrovin-Frobenius manifold structure requires to overcome some additional technical difficulties with respect 
 to the standard procedure of \cite{du97} due to the non regularity of the associated flat pencil.
\section{A flat pencil of metrics associated with $B_n$}
The goal of this Section and of the next Section is to construct a Dubrovin-Forbenius structure on the orbit space of $B_n$, generalizing the ones previously computed for $B_2, B_3$ and $B_4$, that lead to prepotentials containing logarithmic terms. The starting point is the intersection form \eqref{eq:c1}; taking the Lie derivative of $g^{ij}$ with respect to the {\em second highest} degree invariant polynomial $u_{n-1}$ we build a new bilinear form $\eta$ and we prove that the pair $(g, \eta)$ forms a flat pencil of metrics, which is also exact, homogenous and satisfies the Egorov property. By Dubrovin's general correspondence between such pencils and Dubrovin-Frobenius structures, this will allow us to equip the orbit space $\mathbb{C}^n/B_n$ with the latter structure. 

First we will prove a few preliminary results concerning $g$, which, in this set up, plays the role played by the Euclidean cometric in the standard one.
To this end, we will start observing that, as in the Euclidean case, $g$ is $B_n$ invariant and flat.
 
\subsection{Invariance of $g$ with respect to the action of $B_n$} 
First we observe that
\begin{lemma}
The metric defined by $g_{ij}=\left( \frac{1}{n-1}-\delta_{ij}\right)p_ip_j$ and the cometric defined by $g^{ij}=\frac{(1-\delta^{ij})}{p_ip_j}$ are inverse to each other. 
\end{lemma}
\begin{proof}
First we consider $g^{ki}g_{ik}$ (sum over $i$, $k$ fixed) and we get:
$$g^{ki}g_{ik}=\sum_{i=1}^n\left(\frac{1}{n-1}-\delta_{ki}\right) (1-\delta^{ik})\frac{p_ip_k}{p_ip_k}=\sum_{i=1}^n\left(\frac{1}{n-1}-\delta_{ki}\right)(1-\delta^{ik})=$$
$$\sum_{i, i\neq k}\left(\frac{1}{n-1}-\delta_{ki}\right)=1.$$

Next we consider $g^{ki}g_{il}$ (sum over $i$, while $k$ and $l$ are fixed, $k\neq l$) and we get:
$$g^{ki}g_{il}=\sum_{i=1}^n\left( \frac{1}{n-1}-\delta_{ki}\right)\frac{p_l}{p_k}(1-\delta^{il})=\sum_{i,i\neq l}\left( \frac{1}{n-1}-\delta_{ki}\right)\frac{p_l}{p_k}=$$
$$=\left(\sum_{i, i\neq l,k}\frac{1}{n-1}\frac{p_l}{p_k}\right)+\left(\frac{1}{n-1}-1 \right)\frac{p_l}{p_k}=\left(\frac{n-2}{n-1}-\frac{n-2}{n-1}\right)\frac{p_l}{p_k}=0.$$
\end{proof}

The next proposition shows that the metric defined by the $g_{ij}(p)$s introduced in the previous lemma is invariant under the action of $B_n$. Of course, from this the invariance of the corresponding cometric follows. 

\begin{proposition}\label{Lemma:inv}
The metric $g:=g_{ij}(p)dp_i\otimes dp_j=\left( \frac{1}{n-1}-\delta_{ij}\right)p_ip_jdp_i\otimes dp_j$ is invariant under the action of $B_n$ on $V=\mathbb{R}^n$. 
\end{proposition}

\begin{proof}
The action of $B_n$ on $V$ is generated by reflections with respect to the hyperplanes $\{p_j=0\}, j=1,\dots, n$ and $\{p_i\pm p_j=0\}, i,j=1,\dots, n, i<j.$ We denote by $A_{p_j}$ the Jacobian of the transformation associated to the reflection with respect to the hyperplane $\{p_j=0\}$, and analogously for $A_{p_i\pm p_j}$.

The matrix $A_{p_j}$ is a constant diagonal matrix with $1$s on the main diagonal except in position $(j,j)$ where there is $-1$. 
Under the action of the reflection with respect to the hyperplane $\{p_j=0\}$, the metric transforms as $(A_{p_j})^TgA_{p_j}(p=\tilde p)$ where $g$ is the matrix associated to the metric, $T$ denotes transposition and $p=\tilde p$ means that after the matrix operations have been completed, the metric is rewritten in terms of the new coordinates $p_i=\tilde p_i$ for $i\neq j$ and $p_j=-\tilde p_j$. Now it is immediate to see that the action of $A_{p_j}$ on $g$ is to change the sign of all terms that contain $p_j$ except the diagonal term $(\frac{1}{n-1}-1)p_j^2$. Then once it is rewritten in terms of the coordinates $\tilde p$, the metric coincides with the original one. 

As for the reflections with respect to the hyperplanes $\{p_i-p_j=0\}$ we argue as follows. The matrix $A_{p_i-p_j}$ is a constant matrix with $1$s on the main diagonal, except in position $(i,i)$ and $(j,j)$ where there is zero and it has $1$ in position $(i,j)$ and $(j,i)$, while all the other entries are zero. Notice that $A_{p_i-p_j}^T=A_{p_i-p_j}$ and that $A_{p_i-p_j}$ is the matrix representation of a transposition. Therefore, when $A_{p_i-p_j}$ acts on the left on a column vector, it exchanges the positions of $i$-th and $j$-th components of the column vector but it leaves the other unchanged. Similarly, when $A_{p_i-p_j}$ acts on the right on a row vector, it exchanges the positions of $i$-th and $j$-th components of the row vector but it leaves the other unchanged. Thus, $A_{p_i-p_j}^T gA_{p_i-p_j}=A_{p_i-p_j}gA_{p_i-p_j}$ is obtained from $g$ first exchanging the $i$-th and $j$-th rows and then exchanging the $i$-th and $j$-th columns (or first working with the columns and then with the rows) and leaving the rest unchanged. By the form of the columns and rows of $g$, after performing the change of variables $p_k=\tilde p_k$ $k\neq i,j$, $p_i=\tilde p_j$ and $p_j=\tilde p_i$, $A_{p_i-p_j}^T gA_{p_i-p_j}$ coincides with $g$.

Reflections with respect to the hyperplane $\{p_i+p_j=0\}$ are obtained as composition of reflections with respect to the hyperplanes $\{p_i=0\}$, $\{p_j=0\}$ and $\{p_i-p_j=0\}$. To see this, just observe that the matrix $A_{p_i+p_j}$  is a constant  matrix with $1$s on the main diagonal except in positions $(j,j)$ and $(i,i)$ where there is $0$, and it has $-1$ in positions $(i,j)$ and $(j,i)$. Therefore $A_{p_i+p_j}=A_{p_i}A_{p_j}A_{p_i-p_j}$. Now invariance follows from the previous paragraphs. The proposition is proved.

\end{proof}

 Recall that the elementary symmetric polynomials $f_1,\dots,f_n$, in the variables $t_1,\dots,t_n$, are defined by
\begin{equation*}
f_k=\sum_{1 \leqslant i_1<\dots<i_k \leqslant n}  t_{i_1}\cdots t_{i_k},\,k=1,\dots,n.
\end{equation*}

Let $u_0:=1$, $u_k:=0,\,\forall k\geq n+1$ and 

\begin{equation}
u_i:=f_i(p_1^1,\dots,p_n^2),\,i=1,\dots,n.\label{eq:defu}
\end{equation}

The previous result implies the following 

\begin{lemma}\label{cor:invg}
The cometric $g^{ij}(u):=g^{kl}(p)\frac{\partial u_i}{\partial p_k}\frac{\partial u_j}{\partial p_l}$ can be written in terms of the invariant polynomials and it is well-defined on the quotient. Moreover, for each $i$ and $j$, $g^{ij}(u)$ is a homogeneous polynomial in the $p$-variables of degree $2i+2j-4$, which 
depends at most linearly on $u_{n-1}$. In particular, 
\beq\label{g^{11}eq}
g^{11}(u)=4(n^2-n).
\eeq
\end{lemma}
\begin{proof}
The homogeneity of the $g^{ij}(u)$s, as functions of the $p$-variables, is clear. Since all invariant polynomials are really polynomials in $p_1^2, \dots, p_n^2$ no matter which ones we choose, then $\frac{\partial u_j}{\partial p_k}$ contains a factor $p_k$ that cancels the factor $p_k$ in the denominator of $g^{kl}(p)$ and similarly for  $\frac{\partial u_i}{\partial p_l}.$ Thus $g^{ij}(u)$ has entries that are polynomials in the $p$-variables, and since it is invariant by Proposition \ref{Lemma:inv}, it can be written in terms of the invariant polynomials, and thus it is well-defined on the quotient. 

 As $u_i$ is a homogeneous polynomial in the $p$-variables of degree  ${\rm deg}(u_i)=2i$ and, for $k\neq l$, ${\rm deg}(g^{kl}(p))=-2$, see \eqref{eq:defu}, then 
\begin{equation}
 {\rm deg}(g^{ij}(u))=2i-1+2j-1-2=2(i+j)-4,\label{eq:degu}
\end{equation}
as function of the $p$-variables. 
 
For the $g^{ij}(u)$s above the anti-diagonal, i.e. for $i+j<n+1$, therefore we have ${\rm deg}(g^{ij}(u))=2(i+j)-4<2(n+1)-4=2(n-1)$, so those entries can not depend on $u_{n-1}$. All the entries with $(i,j)$ such that $n+1\leq i+j <2n$ depend at most linearly on $u_{n-1}$, since in this range we have $2n-2\leq {\rm deg}(g^{ij}(u)) <4n-4.$ Finally, since $u_n=(p_1\cdots p_n)^2,$ it is immediate to see that each term in the sum (over $k$ and $l$) $g^{nn}(u)=g^{kl}(p)\frac{\partial u_n}{\partial  p_l}\frac{\partial u_n}{\partial p_k}$ contains $u_n$. Since ${\rm deg}(u_n)=2n$ and ${\rm deg}(g^{nn}(u))=4n-4$, we can write $g^{nn}(u)=u_n f$, where $f$ is  polynomial in $p$ of degree $2n-4$, so $f$ can not contain $u_{n-1}$. This proves the claim. 

Now $$g^{11}(u)=g^{kl}(p)\frac{\partial u_1}{\partial p_k}\frac{\partial u_1}{\partial p_l}=\sum_{k,l=1,\dots, n} \frac{(1-\delta^{kl})}{p_kp_l}2p_k2p_l=$$
$$4\sum_{k,l=1, \dots, n}(1-\delta^{kl})=4(n^2-n),$$
thus proving \eqref{g^{11}eq}.
\end{proof}

\subsection{Flatness of $g$}

 Recall that the Christoffel symbols of the Levi-Civita connection $\nabla$ defined by the metric $g$ are the (locally defined) functions
 \begin{equation}
\Gamma^k_{ij}=\frac{1}{2}\sum_{m=1}^ng^{mk}\left(\frac{\partial g_{im}}{\partial p^j}+\frac{\partial g_{jm}}{\partial p^i}-\frac{\partial g_{ij}}{\partial p^m}\right),\label{eq:c2}
\end{equation}
and that the contravariant components of $\nabla$ are 
\begin{equation}
\Gamma^{ij}_{k}(p):=-\sum_{s=1}^ng^{is}(p)\Gamma^j_{sk}(p),\;i,j,k=1,\dots,n.\label{eq:c3}
\end{equation}
Let $g$ be defined as in \eqref{eq:c1}. Then

\begin{lemma}\label{Lemma:ChrisSymbs}
 One has that
\begin{equation}
\Gamma^i_{ii}(p)=\frac{1}{p^i}\quad\text{\rm and}\quad \Gamma^{k}_{ij}(p)=0\quad\text{\rm otherwise}.\label{c3}
\end{equation}
\end{lemma}
\begin{proof} In the following proof all the metric coefficients and all Christoffel symbols depend only on the $p$-variables. To prove \eqref{c3}, first one computes
\begin{align*}
\frac{\partial g_{im}}{\partial p_j}&\stackrel{\eqref{eq:c1}}{=}\frac{\partial}{\partial p_j}\left[\left(\frac{1}{n-1}-\delta_{im}\right)p_ip_m\right]\\
&=\left(\frac{1}{n-1}-\delta_{im}\right)(\delta_{ji}p_m+\delta_{jm}p_i)\\
&=g_{im}\left(\frac{\delta_{ij}}{p_i}+\frac{\delta_{mj}}{p_m}\right).
\end{align*}
This yields
\begin{align*}
g^{mk}\left(\frac{\partial g_{im}}{\partial p_j}+\frac{\partial g_{jm}}{\partial p_i}-\frac{\partial g_{ij}}{\partial p^m}\right)=g^{mk}\left[g_{im}\left(\frac{\delta_{ij}}{p_i}+\frac{\delta_{mj}}{p_m}\right)+g_{jm}\left(\frac{\delta_{ij}}{p_j}+\frac{\delta_{mi}}{p_m}\right)-g_{ij}\left(\frac{\delta_{im}}{p_i}+\frac{\delta_{mj}}{p_j}\right)\right],
\end{align*}
which inserted in \eqref{eq:c2} gives
\begin{align*}
\Gamma^k_{ij}&=\frac{\delta_{ij}}{2}\left[\frac{1}{p_i}\sum_{m=1}^ng^{mk}g_{im}+\frac{1}{p_j}\sum_{m=1}^ng^{mk}g_{jm}\right]\\
&+\frac{1}{2}\left[\sum_{m=1}^ng^{mk}g_{im}\frac{\delta_{mj}}{p_m}+\sum_{m=1}^ng^{mk}g_{jm}\frac{\delta_{mi}}{p_m}\right]\\
&-\frac{g_{ij}}{2}\left[\frac{1}{p_i}\sum_{m=1}^ng^{mk}\delta_{im}+\frac{1}{p_j}\sum_{k=1}^ng^{mk}\delta_{mj}\right]\\
&=\frac{\delta_{ij}}{2}\left(\frac{\delta_{ik}}{p_i}+\frac{\delta_{kj}}{p_j}\right)+\frac{1}{2}\left(\frac{g^{jk}g_{ij}}{p_j}+\frac{g^{ik}g_{ij}}{p_i}\right)-\frac{g_{ij}}{2}\left(\frac{g^{ik}}{p_i}+\frac{g^{jk}}{p_j}\right)
\end{align*}
i.e.
\[
\Gamma^{k}_{ij}=\frac{\delta_{ij}}{2}\left(\frac{\delta_{ik}}{p_i}+\frac{\delta_{jk}}{p_j}\right),
\]
which entails the thesis.
\end{proof}

\begin{proposition}\label{prop:gflatness}
The metric $g_{ij}$ is flat. 
\end{proposition}
\begin{proof}
This can be proved by direct computation of the Riemann tensor using the Christoffel symbols \eqref{c3}. A quicker way to do this is to introduce the connection $1$-form $\omega^i_j:=\Gamma^i_{jk}dp^k$ and the corresponding curvature $2$-forms $\Omega^i_j:=d\omega^i_j +\omega^i_k\wedge \omega^k_j.$
Due to \eqref{c3} we have that $\omega^i_j=0$ if $i\neq j$ and $\omega^i_i=\frac{dp^i}{p^i}=d(\log(p^i))$, which imply $\Omega^i_j=0$, if $i\neq j$ and $\Omega^i_i=\omega^i_i\wedge \omega^i_i=\frac{dp^i\wedge dp^i}{(p^i)^2}=0$ (no sum over $i$) otherwise. 
So the curvature two-form is identically vanishing, which implies that the Riemann tensor vanishes too. 
A third way to prove the flatness of $g$ is to observe that the connection defined by \eqref{c3} is a logarithmic connection with weights that are invariant under the action of $B_n$ (see Example 2.5 in \cite{CHL}). Finally one can prove flatness checking that in the coordinates $(p_1^2,\dots,p_n^2)$ the metric $g$ becomes constant.
\end{proof}

\subsection{Definition of $\eta$} 
In this subsection we introduce $\eta$ as a Lie derivative with respect to the second highest degree invariant polynomial  of the cometric $g^{ij}(u)$, see Lemma \ref{cor:invg}. 
From this, some essential properties of the bilinear form $\eta$ will follow.

\begin{proposition}\label{prop:eta}
The Lie derivative with respect to the vector field $\frac{\partial}{\partial u_{n-1}}$ of the intersection form $g^{ij}(u)$ 
is given by the formula 
\begin{equation}
\eta ^{ij}(u)=\frac{\partial g^{ij}}{\partial u_{n-1}}(u)=4(2n-i-j)u_{i+j-n-1}.\label{eq:eta}
\end{equation}
Hence, $\eta ^{ij}(u)$ is a non-degenerate Hankel matrix with all vanishing entries above the anti-diagonal.
In particular, the entries of the anti-diagonal $i+j=n+1$ are  
\begin{equation*}
\eta ^{i,n-i+1}(u)=4(n-1).
\end{equation*}
\label{formula}g
\end{proposition}
\begin{proof}
If
\begin{equation}
h(x)=\sum_{k=0} ^n u_k x^{n-k} = \prod _{l=1} ^n (x+p_l ^2)\label{eq:hdef}
\end{equation}
and 
\begin{equation*}
{g}^{ij}(u)=\sum_{s,k=1} ^n \frac{(1-\delta^{sk})}{p_s p_k}\dfrac{\partial u_i}{\partial p_s} \dfrac{\partial u_j}{\partial p_k},
\end{equation*}
one has
\begin{align*}
\dfrac{1}{4} \sum_{i,j=1} ^n g^{ij}(u) x^{n-i} y^{n-j}&=\dfrac{1}{4} \sum_{i,j=1} ^n \sum_{s,k=1} ^n \frac{(1-\delta^{sk})}{p^s p^k}\dfrac{\partial u_i}{\partial p_s} \dfrac{\partial u_j}{\partial p_k} x^{n-i} y^{n-j}\\
&=\dfrac{1}{4}  \sum_{s,k=1} ^n \frac{(1-\delta^{sk})}{p_s p_k} \dfrac{\partial}{\partial p_s} \biggl( \sum_{i=1} ^n  u_i x^{n-i} \biggl) \dfrac{\partial}{\partial p_k} \biggl ( \sum_{j=1} ^n  u_j y^{n-j} \biggl)\\
&\stackrel{u_0=1}{=}\dfrac{1}{4}  \sum_{s,k=1} ^n \frac{(1-\delta^{sk})}{p_s p_k} \dfrac{\partial}{\partial p_s} \biggl ( \sum_{i=0} ^n  u_i x^{n-i} \biggl ) \dfrac{\partial}{\partial p_k} \biggl ( \sum_{j=0} ^n  u_j y^{n-j} \biggl )\\
&\stackrel{\eqref{eq:hdef}}{=}\dfrac{1}{4}  \sum_{s,k=1} ^n \frac{(1-\delta^{sk})}{p_s p_k} \dfrac{\partial h(x)}{\partial p_s}  \dfrac{\partial h(y)}{\partial p_k}.
\end{align*}
Since
\begin{equation*}
\dfrac{\partial h(x)}{\partial p_s}= \dfrac{\partial}{\partial p_s} \prod_{l=1} ^n (x+p_l ^2) = 2p_s \prod_{l\neq s} ^n (x+p_l ^2 ),
\end{equation*}
\begin{equation*}
\begin{split}
 \dfrac{1}{4}  \sum_{s,k=1} ^n \frac{1}{p_s p_k} \dfrac{\partial h(x)}{\partial p_s}  \dfrac{\partial h(y)}{\partial p^k}&= \sum_{s,k=1} ^n \prod_{l\neq s} ^n (x+p_l ^2 )  \prod_{q\neq k} ^n (y+p_q ^2 )  \\
&= \sum_{s=1} ^n \prod_{l\neq s} ^n (x+p_l ^2 )  \biggl( \sum_{k=1} ^n \prod_{q\neq k} ^n (y+p_q ^2 ) \biggl)\\
&=h'(x) h'(y) \\
\end{split}
\end{equation*}
and
\begin{equation*}
\begin{split}
- \dfrac{1}{4}  \sum_{s,k=1} ^n \frac{\delta^{sk}}{p_s p_k} \dfrac{\partial h(x)}{\partial p_s}  \dfrac{\partial h(y)}{\partial p_k}&=- \dfrac{1}{4}  \sum_{k=1} ^n \frac{1}{p_k^2} \dfrac{\partial h(x)}{\partial p_k}  \dfrac{\partial h(y)}{\partial p_k}   \\
&=- \sum_{k=1} ^n \prod_{l\neq k} ^n (x+p_l ^2) \prod_{q \neq k} ^n (x+p_q ^2)\\
&=- \sum_{k=1} ^n 
\frac{h(x)h(y)}{(x+p_k ^2)(y+p_k ^2)}  \\
&= - \sum_{k=1} ^n \frac{-h(x)h(y)}{(x-y)(x+p_k ^2)}+ \frac{h(x)h(y)}{(x-y)(y+p_k ^2)}\\
&= -\frac{1}{x-y} \bigg( - \bigg( \sum_{k=1} ^n \frac{h(x)}{x+p_k ^2} \bigg) h(y) + \bigg( \sum_{k=1} ^n \frac{h(y)}{y+p_k ^2} \bigg) h(x) \bigg) \\
&= -\frac{h'(y) h(x)-h'(x) h(y)}{x-y},  \\
\end{split}
\end{equation*}
which yield 
\begin{equation*}
\dfrac{1}{4} \sum_{i,j=1} ^n g^{ij}(u) x^{n-i} y^{n-j}=h'(x) h'(y)-\frac{h'(y) h(x)-h'(x) h(y)}{x-y}.
\end{equation*}
Deriving both sides of the previous identity with respect to $u_{n-1}$ we obtain 
\begin{equation*}
\begin{split}
&  \dfrac{1}{4} \sum_{i,j=1} ^n \frac{\partial g^{ij}}{\partial u_{n-1}}(u) x^{n-i} y^{n-j}= \frac{\partial }{\partial u_{n-1}} \bigg( h'(x) h'(y)-\frac{h'(y) h(x)-h'(x) h(y)}{x-y} \bigg) \\
&  = x^0 \sum_{k=0} ^{n-1} u_k y^{n-k-1} +  y^0 \sum_{k=0} ^{n-1} u_k x^{n-k-1} - \frac{1}{x-y}\bigg(
-x^0 \sum_{k=0} ^{n} u_k y^{n-k} - y \sum_{k=0} ^{n-1} u_k x^{n-k-1} \\
& + y^0 \sum_{k=0} ^{n} u_k x^{n-k} + x \sum_{k=0} ^{n-1} u_k y^{n-k-1} \bigg) \\
& =h'(y)+ h'(x) - \frac{1}{x-y}\bigg(-h(y)-y h'(x)+h(x)+x h'(y) \bigg)\\
&= \frac{h(x)-h(y)+xh'(x)-yh'(y)}{x-y}.
\end{split}
\end{equation*}

Now we have to identify the entries of the matrix $\eta^{ij}(u)=\frac{\partial g^{ij}}{\partial u_{n-1}}(u)$ in the above expression. Deriving $k$ times with respect $x$ we have
\begin{equation*}
\begin{split}
&  \dfrac{1}{4} \sum_{i,j=1} ^n \eta ^{ij}(u) \frac{\partial ^k x^{n-i}}{\partial x^k} y^{n-j}
= \dfrac{1}{4} \sum_{i=1} ^{n-k} \sum_{j=1} ^n \frac{(n-i)!}{(n-i-k)!} \eta ^{ij}(u)  x^{n-i-k} y^{n-j}.
\end{split}
\end{equation*}
Evaluating at $x=0$ we obtain the term that does not depend on $x$, namely the term $i=n-k$
\begin{equation*}
\dfrac{1}{4} k! \sum_{j=1} ^n  \eta ^{n-k,j}(u)  y^{n-j}=\frac{\partial ^k}{\partial x^k} \bigg( \frac{h(x)-h(y)+xh'(x)-yh'(y)}{x-y} \bigg)  \bigg\vert _{x=0}.
\end{equation*}

Similarly, deriving $s$ times with respect $y$ we have
\begin{equation*}
\dfrac{1}{4} k! \sum_{j=1} ^{n-s} \frac{(n-j)!}{(n-j-s)!} \eta ^{n-k,j}(u)  y^{n-j-s}.
\end{equation*}
Evaluating at $y=0$ we obtain the term $j=n-s$, hence
\begin{equation*}
\dfrac{1}{4} k! s!  \ \eta ^{n-k,j-s}(u) =\frac{\partial ^{k+s}}{\partial x^k \partial y^s} \bigg( \frac{h(x)-h(y)+xh'(x)-yh'(y)}{x-y} \bigg)  \bigg\vert _{x=y=0}.
\end{equation*}
Now, we can find each entries of the matrix $\eta ^{ij}(u)$. The lemma is proved.
\end{proof}

\begin{remark}\label{remark:g11}
From now on, since we want $\eta^{i, n-i+1}(u)=1$ for all $i$, we normalize the cometric $g^{ij}$ dividing it by $4(n-1)$. Thus,  using \eqref{g^{11}eq} we have that $g^{11}(u)=n$. 
\end{remark}

A matrix like $\eta^{ij}$ as determined in formula \eqref{eq:eta} is called lower anti-triangular. 
Since the form $\eta$ defined in \eqref{eq:eta} depends polynomially on the $u$s and its determinant is a constant different from zero, we have that
\begin{lemma}
The metric $\eta^{-1}$ depends polynomially on the $u'$s as well. Moreover, $\eta_{ij}$ is also lower anti-triangular. 
\end{lemma}
\proof
Let $p_{\eta}(\lambda)$ be the characteristic polynomial of the matrix associated to the intersection form. It is a polynomial in $\lambda$ with coefficients that are polynomials in the entries of the intersection form and thus they are polynomials in the $u$s. By Cayley-Hamilton theorem, $p_{\eta}(\eta)=0$ identically, where  by $\eta$ we mean the matrix associated to the intersection form. 
But $p_{\eta}(\eta)=\eta^n+c_{n-1}\eta^{n-1}+\dots+c_1 \eta+c_0 1$, where $c_0=(-1)^n \det(\eta)$ and $1$ denotes the identity matrix. 
From this we get immediately
$$\eta^{-1}=\frac{(-1)^{n-1}}{\det(\eta)}(\eta^{n-1}+c_{n-1}\eta^{n-2}+\dots+c_1 1),$$
from which it is clear that the entries of $\eta^{-1}$ are polynomials in the $u$s, since $\det(\eta)$ is a constant and all the other terms depend on the $u$s as polynomials. To show that it is also lower anti-triangular, it is enough to observe that every lower anti-triangular matrix can be obtained as a product $LA$ of two matrices, where $L$ is lower triangular and $A$ is the matrix with all ones on the anti-diagonal and zero in the other entries. Furthermore, it is well-known that the inverse of a lower triangular matrix is lower triangular while the inverse of $A$ coincides with $A.$ This immediately shows that $\eta^{-1}$ is also lower anti-triangular.  
\endproof

\subsection{The pair $(g, \eta)$ is a flat pencil of metrics}

Recall that 
\begin{definition}
A pair of metrics $(g_1,g_2)$ forms a flat pencil if
\begin{itemize}
\item $g=g_1+\lambda g_2$ is a flat metric for all $\lambda$;
\item The Christoffel symbols $\Gamma^{ij}_k$ of the metric $g$ are of the form
\[
\Gamma^{ij}_{k}=\Gamma^{ij}_{1\;k}+\lambda\Gamma^{ij}_{2\;k},\,\forall i,j,k=1,\dots n,\;\forall\lambda.
\]
\end{itemize}
\end{definition}
In this subsection we will show that the pair $(g,\eta)$, where $g$ and $\eta$ are defined in \eqref{eq:c1} and, respectively, in \eqref{eq:eta}, gives rise to a flat pencil of metrics on ${\mathbb{C}^n}/B_n$. Our proof is based on the following result.
\begin{proposition}[Lemma 1.2. in \cite{DCoxeter}]\label{pro:DubCox}
If for a flat metric $g$ on some coordinate system $x=(x_1,\dots,x_n)$ both the components $g^{ij}(x)$ of the metric $g$ and the contravariant components $\Gamma^{ij}_k(x)$ of the associated Levi-Civita connection depend at most linearly on the variable $x_1$, then $g_1:=g$ and $g_2$ defined by
\[
g_2^{ij}(x):=\partial_{x_1}g^{ij}(x),\,\forall i,j,
\]
form a flat pencil if $\det (g^{ij}_2(x))\neq 0$. The contravariant components of the corresponding Levi-Civita connections are
\[
\Gamma^{ij}_{1\;k}(x):=\Gamma^{ij}_k(x)\quad\text{and}\quad \Gamma^{ij}_{2\;k}(x):=\partial_{x_1}\Gamma^{ij}_{k}(x),\;\forall i,j,k.
\]
\end{proposition}

As a system of coordinates on ${\mathbb{C}^n}/B_n$ we choose the set of basic invariants $(u_1,\dots,u_n)$, see \eqref{eq:defu}. Under this assumption, Lemma \ref{cor:invg} entails that the metric defined in \eqref{eq:c1} descends to a metric on the quotient space having the properties required in the Proposition \ref{pro:DubCox}, where the role of $x_1$ is played by $u_{n-1}$. To conclude the proof, we are left to prove that the contravariant components $\Gamma^{ij}_k(u)$ of the Levi-Civita connection defined by $g$ satisfy the conditions stated in Proposition \ref{pro:DubCox}. More precisely we will prove that
\begin{proposition}\label{pro:contra}
The contravariant component of the Levi-Civita connection defined by $g$ are polynomial functions of $(u_1,\dots,u_n)$ which depend at most linearly on $u_{n-1}$.
\end{proposition}

We split the proof of this proposition in two lemmata.

\begin{lemma}\label{poly1.lemma}
In the coordinates $(u_1,\dots,u_n)$ the contravariant components of the Levi-Civita connections defined by $g$ are polynomial functions of $(u_1,\dots,u_n)$.  
\end{lemma}
\proof In the following, unless differently stated, we will sum over repeated indexes. If $\Gamma^i_{jk}(p)$ are the Christoffel symbols in the $p$-variables and $\Gamma^i_{jk}(u)$ those in the $u$-variables, one has
\begin{equation}\label{trans1eq}
\Gamma^l_{ij}(p)=\frac{\partial p^l}{\partial u^c}\frac{\partial^2 u^c}{\partial p^i\partial p^j}+\frac{\partial p^l}{\partial u^c}\frac{\partial u^a}{\partial p^i}\frac{\partial u^b}{\partial p^j}\Gamma^c_{ab}(u).
\end{equation}

Multiplying both sides of \eqref{trans1eq} by $g^{ki}(p)\frac{\partial u^f}{\partial p^k}\frac{\partial u^d}{\partial p^l}dp^j$, we obtain 
\begin{align*}
g^{ki}(p)\frac{\partial u^f}{\partial p^k}\frac{\partial u^d}{\partial p^l}\Gamma^l_{ij}(p)dp^j&=g^{ki}(p)\frac{\partial u^f}{\partial p^k}\frac{\partial u^d}{\partial p^l}\frac{\partial p^l}{\partial u^c}\frac{\partial^2 u^c}{\partial p^i\partial p^j}dp^j\\
&+g^{ki}(p)\frac{\partial u^f}{\partial p^k}\frac{\partial u^d}{\partial p^l}\frac{\partial p^l}{\partial u^c}\frac{\partial u^a}{\partial p^i}\frac{\partial u^b}{\partial p^j}\Gamma^c_{ab}(u)dp^j.
\end{align*}
Now observe that in the two terms of the right-hand side of the above expression $\frac{\partial u^d}{\partial p^l}\frac{\partial p^l}{\partial u^c}=\delta^d_c$, so it simplifies to:
\[
g^{ki}(p)\frac{\partial u^f}{\partial p^k}\frac{\partial u^d}{\partial p^l}\Gamma^l_{ij}(p)dp^j=g^{ki}(p)\frac{\partial u^f}{\partial p^k}\frac{\partial^2 u^d}{\partial p^i\partial p^j}dp^j+g^{ki}(p)\frac{\partial u^f}{\partial p^k}\frac{\partial u^a}{\partial p^i}\Gamma^d_{ab}(u)du^b.
\]

Using the definition of Christoffel symbols with two upper indices we get:
\[-\frac{\partial u^f}{\partial p^k}\frac{\partial u^d}{\partial p^l}\Gamma^{kl}_j(p)dp^j=g^{ki}(p)\frac{\partial u^f}{\partial p^k}\frac{\partial^2 u^d}{\partial p^i\partial p^j}dp^j-\Gamma^{fd}_{b}(u) du^b,
\]
where we have used the fact that $g^{ki}(p)\frac{\partial u^f}{\partial p^k}\frac{\partial u^a}{\partial p^i}$ is the cometric written in the $u$-variables. We thus obtain 
\begin{equation}\label{trans2.eq}
\Gamma^{fd}_b(u) du^b =g^{ki}(p)\frac{\partial u^f}{\partial p^k}\frac{\partial^2 u^d}{\partial p^i \partial p^j}dp^j+\frac{\partial u^f}{\partial p^k}\frac{\partial u^d}{\partial p^l}\Gamma^{kl}_j(p)dp^j.
\end{equation}
Introducing the contravariant Christoffel symbols $\Gamma^{ik}_l(p)=-g^{im}(p)\Gamma^k_{ml}(p)$,
 from \eqref{c3} one obtains 
\[
\Gamma^{ik}_l(p)\stackrel{\eqref{eq:c3}}{=}-\frac{(1-\delta^{im})}{p^ip^kp^m}\delta_{km}\delta_{kl}=\frac{(\delta^{ki}-1)\delta_{kl}}{p^i (p^k)^2},
\]
which, inserted in \eqref{trans2.eq}, yields
\begin{equation}\label{trans3.eq}
\Gamma^{fd}_b(u)du^b=\frac{(1-\delta^{ki})}{p^i p^k}\frac{\partial u^f}{\partial p^k}\frac{\partial^2 u^d}{\partial p^i \partial p^j}dp^j+\frac{\partial u^f}{\partial p^k}\frac{\partial u^d}{\partial p^l}\frac{(\delta^{kl}-1)\delta_{lj}}{p^k (p^l)^2}dp^j.
\end{equation}
Expanding and analyzing the right-hand side of \eqref{trans3.eq},
we obtain: 
\[ \sum_{k,i,j, k\neq i} \frac{1}{p^ip^k}\frac{\partial u^f}{\partial p^k}\frac{\partial^2 u^d}{\partial p^i \partial p^j}dp^j-\sum_{k,l,j,k\neq l}\frac{\partial u^f}{\partial p^k}\frac{\partial u^d}{\partial p^l}\frac{\delta_{lj}}{p^k (p^l)^2}dp^j=
\]
\[=\sum_{k,j,k\neq j}\frac{1}{p^jp^k}\frac{\partial u^f}{\partial p^k}\frac{\partial^2 u^d}{(\partial p^j)^2}dp^j+\sum_{k,i,j,k\neq i, j \neq i}\frac{1}{p^ip^k}\frac{\partial u^f}{\partial p^k}\frac{\partial^2 u^d}{\partial p^i \partial p^j}dp^j-\sum_{k,j,k\neq j}\frac{\partial u^f}{\partial p^k}\frac{\partial u^d}{\partial p^j} \frac{1}{p^k (p^j)^2} dp^j,
\]
which can be written as:
\[ \sum_{k,j,k\neq j}\frac{1}{p^j p^k}\frac{\partial u^f }{\partial p^k}\left [  \frac{\partial^2 u^d}{(\partial p^j)^2}-\frac{\partial u^d}{\partial p^j}\frac{1}{p^j} \right]dp^j +\sum_{k,i,j,k\neq i, j\neq i }\frac{1}{p^i p^k} \frac{\partial u^f}{\partial p^k}\frac{\partial^2 u^d}{\partial p^i \partial p^j} dp^j.\]
Taking into account that
\[
u^k=\sum_{1\leq i_1<\cdots <i_k\leq n} (p^{i_1}\cdots p^{i_k})^2,
\]
it is immediate to check that first term above vanishes identically, since $u^1, \dots, u^n$ are polynomials of degree $1$ in each of the $(p^i)^2$, and the second term does not contain any denominator, since they are simplified (unless $d=1$ in which case the second term is identically zero). The previous (long) discussion is summarized in the following formula
\begin{equation}
\Gamma^{rs}_{b}(u)du^b=\sum_{k,i,j,k\neq i, j\neq i }\frac{1}{p^i p^k} \frac{\partial u^r}{\partial p^k}\frac{\partial^2 u^s}{\partial p^i \partial p^j} dp^j,\;\forall r,s=1,\dots,n,\label{eq:solmon}
\end{equation}
whose right-hand side is a $1$-form with polynomial coefficients in the $p$-variables.

To conclude we can argue as follows. Since the left-hand side of \eqref{eq:solmon} is $B_n$-invariant, the right-hand side is so. Since the latter is a $1$-form with polynomial coefficients, the coefficients of the left-hand side are necessarily polynomial functions in $(u_1,\dots,u_n)$, see \cite[Theorem page 3]{solmon}.
\endproof

\begin{remark}
The previous argument is the same used in the proof Lemma 2.1 in \cite{DCoxeter}. However, while it is evident that the left-hand side of Formula (2.8) in \cite{DCoxeter} is a $1$-form with polynomial coefficients, the polynomiality of the coefficients of the right-hand side of \eqref{eq:solmon} was not so and it needed to be shown.
\end{remark}

To complete the proof of Proposition \ref{pro:contra}, we are left to show that the contravariant components of the Levi-Civita connection of $g$ depend at most linearly on $u_{n-1}$. This result follows from the following  

\begin{lemma}\label{lemma6.13}
For every choice of $s,i,k=1,\dots,n$,
\begin{equation}
\text{deg}(\Gamma^{si}_k(u))<4n-4.\label{eq:disgam}
\end{equation}
\end{lemma}
\proof
First we will show that for every choice of the indices 
\begin{equation}
\text{deg}(\Gamma^{c}_{ab}(u))=\text{deg}(u^c)-\text{deg}(u^a)-\text{deg}(u^b).\label{eq:inGamma}
\end{equation}
To this end, we start noticing that if not all the indices in the left-hand side of \eqref{trans1eq} are equal, \eqref{c3} implies
\[
\frac{\partial p^l}{\partial u^c}\frac{\partial^2 u^c}{\partial p^i\partial p^j}+\frac{\partial p^l}{\partial u^c}\frac{\partial u^a}{\partial p^i}\frac{\partial u^b}{\partial p^j}\Gamma^c_{ab}(u)=0,
\]
which yields
\[
\frac{\partial^2 u^c}{\partial p^i\partial p^j}+\frac{\partial u^a}{\partial p^i}\frac{\partial u^b}{\partial p^j}\Gamma^c_{ab}(u)=0.
\]

This identity, together with the definition of the invariants $u^1,\dots,u^n$, implies that $\Gamma^{c}_{ab}(u)$ is a homogeneous polynomial of degree
\[
\text{deg}(\Gamma^{c}_{ab}(u))=\text{deg}(u^c)-\text{deg}(u^a)-\text{deg}(u^b).
\]
On the other hand, if in \eqref{trans1eq} $i=j=l$, \eqref{c3} entails

\[
\frac{\partial u^c}{\partial p^i}\frac{1}{p^i}=\frac{\partial^2 u^c}{\partial^2 p^i}+\frac{\partial u^a}{\partial p^i}\frac{\partial u^b}{\partial p^i}\Gamma^c_{ab}(u),
\]
which implies that
\[
\text{deg}(u^c)-2=\text{deg}(\Gamma^{c}_{ab}(u))+\text{deg}(u^a)+\text{deg}(u^b)-2,
\]
or, equivalently, that
\[
\text{deg}(\Gamma^{c}_{ab}(u))=\text{deg}(u^c)-\text{deg}(u^a)-\text{deg}(u^b),
\]
proving \eqref{eq:disgam}. 
To conclude the proof of the Lemma, it suffices to note since $\Gamma_{k}^{si}(u)= -{g}^{sj}(u){\Gamma}_{jk} ^i(u)$ and $\text{deg}(u^i)=2i$, for all $i=1,\dots,n$,
\begin{align*}
\text{deg}\big(\Gamma_{k}^{si}(u)\big)&=\text{deg}\big({g}^{sj}(u)\big) + \text{deg}\big({\Gamma}_{jk}^i(u)\big)\\
&\stackrel{\eqref{eq:degu}}{=}\text{deg}(u^s)+\text{deg}(u^j)-4+\text{deg}(\Gamma^{i}_{jk}(u))\\
&\stackrel{\eqref{eq:inGamma}}{=}\text{deg}(u^s)-4+\text{deg}(u^i)-\text{deg}(u^k)\\
&= 2s +2i -2k -4\leqslant 4n -6 < 4n -4.
\end{align*}
\endproof

\begin{corollary}
Since $\text{deg}(u^{n-1})=2n-2$, it follows from Lemma \ref{lemma6.13} that $\Gamma^{si}_k(u)$, for all $s,k,i=1,\dots, n$, depends at most linearly on $u^{n-1}$. 
\end{corollary}

Summarizing, we have
\begin{theorem}
The pair $(g, \eta)$ gives rise to a flat pencil of metrics.
\end{theorem}
\begin{proof}
The metric $g^{ij}(u)$ is well-defined on the quotient, it depends at most linearly on $u_{n-1}$ by Lemma \ref{cor:invg} and it is flat by Proposition \ref{prop:gflatness}.  Furthermore, its contravariant Christoffel symbols are also polynomial functions that depend at most linearly on $u_{n-1}$ by Proposition \ref{pro:contra}. Therefore, since $\eta^{ij}(u):=\frac{\partial g^{ij}}{\partial u_{n-1}}(u)$ has non-zero constant determinant by Proposition \ref{prop:eta}, $(g, \eta)$ forms a flat pencil of metrics by Proposition \ref{pro:DubCox}.
\end{proof}

We close this subsection with a result which will play a crucial role to prove the existence of a Dubrovin-Frobenius structure on the orbit space ${\mathbb C}^n/B_n$.

\begin{proposition}\label{prop:flatcoo}[Corollary 2.4 in \cite{DCoxeter}]
There exists a set of $B_n$-invariant, homogeneous polynomials $t_1(p),\dots,t_n(p)$, $\text{deg}\;(t_k(p))=2k$ for all $k=1,\dots,n$, such that $\eta^{ij}$ is constant in the coordinates $(t^1,\dots,t^n)$.
\end{proposition}
\proof We will make only a few comments about the proof of this statement, referring the reader, for more details, to \cite{DCoxeter}. In this reference the existence of this set of coordinates was proven for (all Coxeter groups and) $g$ equal to the standard Euclidean metric of $\mathbb R^n$. The proof was based on the following hypothesis, all of them   verified also in our case: the flatness of $\eta$ and the polynomiality of both $\eta^{-1}$ and of the Christoffel symbols of $\eta$, when written in the coordinates defined by any set of invariants $(u^1,\dots,u^n)$ with deg($u^i$)$=d_i$. Under these assumptions it is immediate to check that the Pfaffian system defining the flat coordinates has polynomial coefficients. The statement of the theorem then follows from
\begin{itemize}
\item the analiticity of the solutions of a compatible Pfaffian systems with polynomial coefficients (see for instance \cite{Ha}).
\item the invariance of the space of solutions with respect to scaling transformations 
$u^i\to c^{d_i}u^i$. 
\end{itemize}
\endproof

\begin{lemma}\label{lem:DSflat}
In our case, the flat coordinates of Proposition \ref{prop:flatcoo} can be further chosen so that: 
\begin{equation}
\eta^{ij}(t)=\delta_{i,n+1-j}.\label{eq:etacor}
\end{equation}
The coordinates so defined are called Dubrovin-Saito flat coordinates.
\end{lemma}
\proof By Proposition \ref{prop:flatcoo} flat coordinates for $\eta^{ij}$ are homogenous invariant polynomials with distinct degrees. Therefore, in order to prove the claim of the Lemma, by Corollary 1.1 in \cite{DLectures} it is enough to show that there exists a system of flat coordinates $t^1, \dots, t^n$ such that $\eta_{nn}(t)=0$. 
Consider the contravariant metric $\eta$ written in the $u$-variables, see \eqref{eq:eta}.
Observe that $\eta^{nn}(u)=0$. Recall that 
$\eta_{nn}(u)=\frac{1}{\det(\eta)}\rm{adj}(\eta)_{nn}$, where $\rm{adj}(\eta)$ $=C^T$ and where $C$ is the cofactor matrix of $\eta$, whose entry $(i,j)$ is $(-1)^{i+j}$ times the $(i,j)$ minor of $\eta$. Since $\eta$ is lower anti-triangular, its $(n,n)$ minor is zero, therefore $\eta_{nn}(u)=0$. 
Rewriting $\eta^{-1}$ in a flat coordinate system $(t^1, \dots, t^n)$ we have
$\eta_{ij}(u(t))\frac{\partial u^i}{\partial t^k}\frac{\partial u^j}{\partial t^l} dt^k \otimes dt^l.$
Now 
$$\eta_{nn}(t)=\eta_{ij}(u(t))\frac{\partial u^i}{\partial t^n}\frac{\partial u^j}{\partial t^n}.$$
But since $\frac{\partial u^i}{\partial t^n}=0$ for degree reasons unless $i=n$, and analogously for the other partial derivatives, we have
$$\eta_{nn}(t)=\eta_{nn}(u(t))\left(\frac{\partial u^n}{\partial t^n}\right)^2=0,$$
(no sum over $n$) since $\eta_{nn}(u)=0$.
\endproof

\begin{remark}
It is easy to check that non-vanishing contravariant Christoffel symbols $a^i_{jk}$ of the
 Saito flat metric in the coordinates $(u^1,\dots,u^n)$ are given by
\[a^{ij}_{i+j-n-1}=4(n-j).\]
Using the above formula one can verify that the invariants $u^1$, $\tau$ and $u^n$
 are flat coordinates.
\end{remark}

\begin{remark}\label{remark2:g11}
It is also immediate to verify that, in the Dubrovin-Saito flat coordinates, $g^{11}(t)=n$, up to a possible rescaling by a constant, see Remark \ref{remark:g11}.
\end{remark}

\section{Dubrovin-Frobenius structure on ${\mathbb C}^n /B_n$}
\subsection{From flat pencils of metrics to Dubrovin-Frobenius manifolds}
Flat pencils of contravariant metrics are a key component in the theory of Dubrovin-Frobenius manifolds. 
 More precisely, one can prove that any Dubrovin-Frobenius structure defines a flat pencil of contravariant metrics (see \cite{DLectures}), and, conversely, that a Dubrovin-Frobenius structure can be defined starting from a flat pencil of metrics satisfying the following three additional properties, see \cite{du97} (see also \cite{DLectures} and \cite{ABLR21}):
\begin{itemize}
\item \emph{Exactness}: there exists a vector field $e$ such that
\beq\label{eq:exact}
\cL_e g=\eta,\qquad \cL_e \eta=0,
\eeq
where $\cL_e$ denotes Lie derivative with respect to $e$.
\item \emph{Homogeneity}: 
\beq\label{eq:homogen}
\cL_E g=(d-1)g,
\eeq
where $E^i:=g^{il}\eta_{lj}e^j$.
\item \emph{Egorov property}: locally there exists a function $\tau$ such that
\beq\label{eq:egorov}
e^i=\eta^{is}\d_s\tau,\qquad E^i=g^{is}\d_s\tau.
\eeq
\end{itemize}
Exactness implies that $[e,E]=e$ and combining this with the homogeneity condition one obtains
\begin{equation}
\cL_E \eta=\cL_E\cL_e g=\cL_e\cL_E g-\cL_{[E,e]} g=(d-2)\eta.\label{eq:hometa}
\end{equation}
Moreover, for Dubrovin-Frobenius manifolds the vector fields $e$ and $E$ coincide with the unit vector field and the Euler vector field, respectively.

To prove that the flat pencil $(g,\eta)$ induces a Dubrovin-Frobenius structure on $\mathbb C^n/B_n$, we will start to show that the $(g,\eta)$ is exact, homogeneous and that it satisfies the Egorov property or, using Dubrovin's terminology, that it is quasihomogeneous.

\begin{lemma}
The pair $(g,\eta)$ form an exact pencil. 
\end{lemma}
\begin{proof}
The first of \eqref{eq:exact} is true by definition and the second follows from the fact that $\eta$ does not depend on $u_{n-1}$ as it can be inferred from formula \eqref{eq:eta}.
\end{proof}

Let $e=\frac{\partial}{\partial u_{n-1}}$. Write $\partial_k:=\frac{\partial}{\partial u_k}$ for all $k$. Recall that 
\begin{equation}
\eta^{ij}=4(2n-i-j)u^{i+j-n-1}.\label{tau1}
\end{equation}
Then
\begin{lemma} If $\tau$ is given by 
\begin{equation} 
\tau:=\frac{1}{4(n-1)}\left(u^2-\frac{(n-2)}{2(n-1)}(u^1)^2\right)\label{tau2}
\end{equation}
then
\begin{equation}
e^i=\eta^{ij}\partial_j \tau,\label{tau3}
\end{equation}
so the first of \eqref{eq:egorov} is fulfilled. 
\end{lemma}
\begin{proof}
The proof is by a direct computation.
Using \eqref{tau2} and \eqref{tau1}, one obtains
\begin{align}
e^i&=\frac{1}{4(n-1)}\sum_{j=1}^n\left(\eta^{ij}\delta_{j2}-\frac{(n-2)}{(n-1)}\eta^{ij}\delta_{j1}u^1\right)\nonumber\\
&=\frac{1}{4(n-1)}\eta^{i2}-\frac{(n-2)}{4(n-1)}\eta^{i1}u^1\nonumber\\
&=\frac{(2n-i-2)}{(n-1)}u^{i+1-n}-\frac{(n-2)(2n-i-1)}{(n-1)^2}u^{i-n}u^1.\label{tau4}
\end{align}
Since $u^k=0$ for all $k<0$, if $i<n-1$ both summands in \eqref{tau4} are zero. If $i=n$, \eqref{tau4} becomes
\[
\frac{(2n-n-2)}{(n-1)}u^{n+1-n}-\frac{(n-2)(2n-n-1)}{(n-1)^2}u^{n-n}u^1=\frac{(n-2)}{(n-1)}u^1-\frac{(n-2)}{(n-1)}u^0u^1=0,
\]
since $u^0=1$. Finally, if $i=n-1$, one obtains
\[
\frac{(2n-(n-1)-2)}{(n-1)}u^{n-1+1-n}-\frac{(n-2)(2n-(n-1)-1)}{(n-1)^2}u^{(n-1)-n}u^1=1,
\]
which proves our statement.
\end{proof}

\begin{lemma}
Defining
\begin{equation}
E^i:=g^{ij}\partial_j \tau\label{tau5}
\end{equation}
one has that
\begin{equation}
E^i=g^{il}\eta_{lj}e^j,
\end{equation}
so that the second of \eqref{eq:egorov} is fulfilled. 
\end{lemma}
\begin{proof}
This follows from \eqref{tau3} and from \eqref{tau5}, recalling that $\eta^{ij}\eta_{jl}=\delta^i_l$. 
\end{proof}

One can prove that
\begin{lemma} Writing $\partial_k=\frac{\partial}{\partial p^k}$, one has
\begin{equation}
E=\frac{1}{2(n-1)}\sum_{k=1}^np^k\partial_k.\label{tau6}
\end{equation}
\end{lemma}
\begin{proof}

The proof follows at once from \eqref{eq:defu}, \eqref{tau2} and \eqref{tau5}. First one computes
\begin{equation*}
\partial_j (u^1)^2= 4 p^j u^1\quad\text{and}\quad\partial_j u^2=2p^ju^1-(p^j)^3,
\end{equation*}
which yield
\[
\partial_j\tau=\frac{1}{2(n-1)}\left[\frac{p^ju^1}{n-1}-(p^j)^3\right].
\]
Then
\begin{align*}
E^i=g^{ij}(p)\partial_j\tau&=\frac{1}{2(n-1)}\sum_{j=1}^n\frac{(1-\delta_{ij})}{p^ip^j}\left[\frac{p^ju^1}{n-1}-(p^j)^3\right]\\
&=\frac{1}{2p^i(n-1)}\sum_{j\neq i}\left[\frac{u^1}{n-1}-(p^j)^2\right]\\
&=\frac{1}{2p^i(n-1)}\left[\frac{(n-1)u^1}{n-1}-u^1+(p^i)^2\right]\\
&=\frac{p_i}{2(n-1)}.
\end{align*}
\end{proof}
Recall that $\text{deg}(u_k)=2k$, and that $g^{lk}$ is a homogeneous polynomial of degree $2k+2l-4$ (in the $u$s). From this it follows:
\begin{proposition} We have that
\begin{equation}
\mathcal L_E g=(d-1)g,\label{tau7}
\end{equation}
where $d=1-\frac{2}{(n-1)}$,
therefore condition \eqref{eq:homogen} is fulfilled. 
\end{proposition}
\begin{proof}
First one observes that $\mathcal L_E(du^k)=\frac{k}{(n-1)}du^k$. Then
\begin{align*}
(\mathcal L_Eg)(du^l,du^k)&=\mathcal L_E(g(du^l,du^k))-g(\mathcal L_E du^l,du^k)-g(du^l,\mathcal L_E du^k)\\
=&\mathcal L_E g^{kl}-\frac{l}{(n-1)}g^{lk}-\frac{k}{(n-1)}g^{lk}\\
=&\mathcal L_E g^{kl}-\frac{l+k}{(n-1)}g^{lk}\\
=&\frac{l+k-2}{(n-1)}g^{lk}-\frac{l+k}{(n+1)}g^{lk}\\
=&-\frac{2}{(n-1)}g^{lk}\\
=&-\frac{2}{(n-1)}g(du^l,du^k).
\end{align*}
\end{proof}

Before moving on, we observe that
\begin{remark}\label{rem:hom}
If $(f^1,\dots,f^n)$ is any system of \emph{homogeneous} coordinates in the $p$-variables
\begin{align*}
E=\frac{1}{2(n-1)}\sum_{k=1}^np^k\frac{\partial}{\partial p^k}&=\frac{1}{2(n-1)}\sum_{k=1}^np^k\sum_{j=1}^n \frac{\partial f^j}{\partial p^k}\frac{\partial}{\partial f^j}\\
&=\frac{1}{2(n-1)}\sum_{j=1}^n\left(\sum_{k=1}^np^k\frac{\partial f^j}{\partial p^k}\right)\frac{\partial}{\partial f^j}\\
&=\frac{1}{2(n-1)}\sum_{j=1}^n\text{deg}\;(f^j)f^j\frac{\partial}{\partial f^j}.
\end{align*}
\end{remark}

Our next step in the construction of the Dubrovin-Frobenius structure on $\mathbb C^n/B_n$, will be the introduction of the constant structures defining the relevant product. To this end, recall that a homogeneous flat pencil $(g,\eta)$ on $M$ is called regular if the endomorphism of $TM$ defined by  
\begin{equation}
R^i_j=\nabla^\eta_jE^i-\nabla^g_jE^i,\label{eq:reg}
\end{equation}
is invertible, where, in the previous formula, $\nabla^\eta,\nabla^g$ denote the (covariant derivative operators of the) Levi-Civita connections of the metrics $\eta$ and, respectively, $g$. Under the regularity assumption, the flat pencil defines a structure of a Dubrovin-Frobenius manifold on $M$ whose structure constants are defined by the following formulas
\begin{equation}
c_{hk}^j=L^s_h({\Gamma_\eta}^l_{sk}-{\Gamma_g}^l_{sk})(R^{-1})^j_l\label{eq:c11}
\end{equation}
where $L^s_h=g^{si}\eta_{ih}$, ${\Gamma_\eta}^l_{sk}$ and ${\Gamma_g}^{l}_{sk}$ are the Christoffel's symbols of the metrics $\eta$ and, respectively $g$. From now on, unless explicitly stated, all the tensors will be written in the in the flat Dubrovin-Saito coordinates, see Proposition \ref{prop:flatcoo} and Lemma \ref{lem:DSflat} above. Since in these coordinates ${\Gamma_\eta}^l_{sk}=0$ for all $l,s,k$, in order to keep the notation more readable, we use directly the notation $\Gamma^i_{jk}$ for the Christoffel symbols associated to $g$ (as we did in Section 5.4). Under these assumptions, Formula \eqref{eq:c11} becomes
\begin{equation}
c_{hk}^j=-L^s_h{\Gamma}^l_{sk}(R^{-1})^j_l=-g^{si}\eta_{ih}{\Gamma}^l_{sk}(R^{-1})^j_l\stackrel{\eqref{eq:c3}}{=}\eta_{hi}{\Gamma}^{il}_k(R^{-1})^j_l,\label{eq:c22}
\end{equation} 
see \cite{ABLR21} and references therein. On the other hand, one can prove that the flat pencil of metrics $(g,\eta)$ defined above is not regular. To this end it suffices to note that 
\begin{equation}
R^i_j=\frac{d-1}{2}\delta^i_j+\nabla^\eta_j E^i,\label{eq:c33}
\end{equation}  
see, for example, \cite[Remark 5.7]{ABLR21}, which, in our case, entails
\begin{equation}
R^i_j=\frac{(j-1)}{n-1}\delta^i_j.\label{eq:c4}
\end{equation}
In fact, since {$d=1-\frac{2}{n-1}$, using the Dubrovin-Saito flat coordinates
\[
R^i_j=\frac{d-1}{2}\delta^i_j+\nabla_j^\eta E^i=-\frac{1}{n-1}\delta^i_j+\frac{j}{n-1}\delta^i_j=\frac{(j-1)}{n-1},
\]
see Remark \ref{rem:hom}. In spite our flat pencil of metrics is not regular, we will be able to prove the following
\begin{theorem}\label{thm:maint}
The flat pencil of metrics $g-\lambda \eta$ gives rise to a Dubrovin-Frobenius structure on $\mathbb{C}^n/B_n$ generalizing those computed explicitly for the cases $n=2,3,4$. 
\end{theorem}

The proof of this result will consist of the following steps:
\begin{enumerate}
\item[(i)] Definition of the structure constants of the product.
\item[(ii)] Proof of the commutativity of the product.
\item[(iii)] Existence of a flat unit vector field.
\item[(iv)] Identification of the metric $\eta$ with the invariant metric of the Dubrovin-Frobenius manifold.
\item[(v)] Identification of the cometric $g$ with the intersection form of the Dubrovin-Frobenius manifold.
\item[(vi)] Symmetry of the tensor $\nabla c$.
\item[(vii)] Associativity of the product.
\end{enumerate}
In all steps of the proof we will work in Saito flat coordinates. 
In order to prove the last step we will preliminarly prove that the functions
 \begin{equation}
b^{ij}_k=\big(1+d_j-\frac{d_F}{2}\big)c^{ij}_k,\label{eq:simeta}
\end{equation}
coincide with the contravariant Christoffel symbols of the cometric $g$. This will allows us to obtain part of the associativity conditions as a consequence of the vanishing of the curvature.

We start with a preliminary lemma: 
\begin{lemma} In Saito flat coordinates the contravariant symbols of the Levi-Civita of the metric $g$ satisfy 
\begin{eqnarray}
{\Gamma}^{n+1-h,k}_m&=&{\Gamma}^{n+1-m,k}_h,\label{eq:id1}\\
g^{is}\Gamma^{jk}_s&=&g^{js}\Gamma^{ik}_s,\label{eq:id4}\\
{\Gamma}^{ij}_s{\Gamma}^{sk}_l&=&{\Gamma}^{ik}_s{\Gamma}^{sj}_l,\label{eq:id2}\\
\f{{\Gamma}^{mh}_k}{R^h_h}&=&\f{{\Gamma}^{hm}_k}{R^m_m},\qquad (h,m)\ne (1,1).\label{eq:id3}
\end{eqnarray}
where $\Gamma^{jk}_i$ are the contravariant Christoffel of $g$ in Saito flat coordinates.
\end{lemma}
\begin{proof}
The following identities hold true (see \cite{du97} and \cite{ABLR21}):
\begin{eqnarray}
\eta_{hs}\Delta^{sk}_m&=&\eta_{ms}\Delta^{sk}_h,\\
g^{is}\Delta^{jk}_s&=&g^{js}\Delta^{ik}_s,\\
\Delta^{ij}_s \Delta^{sk}_l &=& \Delta^{ik}_s \Delta^{sj}_l,\\
\Delta^{tl}_k(R^{-1})^s_l&=&\Delta^{sl}_k(R^{-1})^t_l.
\end{eqnarray}
where the tensor $\Delta^{jk}_m$ is given in terms of the Levi-Civita connections $\nabla^{\eta}$ and $\nabla^{g}$ by the  formula
\[
\Delta^{jk}_m=\eta_{lm}\left(\eta^{js}\Gamma^{lk}_{(g)s}-g^{sl}\Gamma^{jk}_{(\eta)s}\right)=\eta_{lm}\left(\eta^{ls}\Gamma^{jk}_{(g)s}-g^{js}\Gamma^{lk}_{(\eta)s}\right).
\]
In Saito flat coordinates $\Gamma^{jk}_{(g)i}=\Gamma^{jk}_{i}$, $\Gamma^{jk}_{(\eta)i}=0$, $\Delta^{jk}_i=\Gamma^{jk}_i$, $\eta_{ij}=\delta_{i,n+1-j}$ and the above identities reduce to identities (\ref{eq:id1},\ref{eq:id4},\ref{eq:id2},\ref{eq:id3}).
\end{proof}

\subsection{Step 1: definition of the $c^{i}_{jk}$s} As we have already mentioned, the definition of the Dubrovin-Frobenius structure on $\mathbb C^n/B_n$ cannot \emph{completely} hinge on \eqref{eq:c11} since the endomorphism $R$ defined in \eqref{eq:c4} in not invertible. On the other hand, the loss of information is restricted to the case $R^i_j=0$, i.e. $i=j=1$, see Formula \eqref{eq:c4}. In this way, Formula \eqref{eq:c22} permits to fix all the $c^{i}_{jk}$s, but the ones with $i=1$. In other words, for all $i\neq 1$
\begin{equation}\label{eq:cons1}
c^{i}_{jk}:=\frac{\eta_{jh}{\Gamma}^{hi}_k}{R^i_i}.
\end{equation}
Note that one has that 
\begin{equation}
c^{i}_{jk}=\frac{{\Gamma}^{n+1-j,i}_k}{R^i_i}\stackrel{\eqref{eq:id1}}{=}\frac{{\Gamma}^{n+1-k,i}_j}{R^i_i}.\label{eq:cons2}
\end{equation}
Both equalities follow since we are working with the Dubrovin-Saito coordinates. In particular, the first equality follows from the form of the metric $\eta$ when written in these coordinates, i.e. $\eta_{ij}=\delta_{i,n+1-j}$, see Lemma \ref{lem:DSflat}. The remaining $c^i_{jk}$s will be defined via the following:
\begin{align}
c^1_{ij}&:=c^{n+1-j}_{ni},\qquad\forall (i,j)\neq (n,n);\label{eq:cons3}\\
c^1_{nn}&:=\frac{(n-1)}{t_n}.\label{eq:cons4}
\end{align}
The structure constants $c^k_{ij}$s defined in \eqref{eq:cons1}, \eqref{eq:cons3} and \eqref{eq:cons4}, are homogeneous polynomials of the $p$-variables of degree $2(n-1+k-i-j)$, see (the end of the proof of) Lemma \ref{lemma6.13}. In particular, note that, with the exception of $c^1_{nn}$, 
\begin{equation}\label{eq:ineq}
c^k_{ij}=0,
\end{equation}
for all $i,j,k$ such that $i+j>n+k-1$. Now we have to prove that the structure constants defined above satisfy all the conditions
 entering the definition of Dubrovin-Frobenius manifolds.

\begin{remark}\label{rem:degnor}
Hereafter we will normalize the degree of the $p$-homogeneous polynomials by $\frac{1}{2(n-1)}$ accordingly with the expression of Euler vector field, see \eqref{tau6}. In other words, we will set
 \begin{equation}\label{eq:degh}
d_k:=\text{deg}\;(f_k)=\frac{k}{n-1},
\end{equation}
where $f_k$ is any degree $2k$, homogeneous polynomial in the $p$-variables. For example
\begin{equation}
d_{n-1+k-i-j}:=\text{deg}\;(c^k_{ij})=\frac{n-1+k-i-j}{n-1},\label{eq:degc}
\end{equation}
and $d_{i+j-2}:=\text{deg}\;(g^{ij}(u))=\frac{i+j-2}{n-1}$, see \eqref{eq:degu}.
\end{remark}

\subsection{Step 2: commutativity of the product}
We have to prove that for all $i,j,k=1,\dots,n$,
\begin{equation}
c^i_{jk}=c^i_{kj}.\label{eq:commc}
\end{equation}
For $i\neq 1$ this follows automatically from \eqref{eq:cons2}. For $i=1$, it follows from \eqref{eq:cons3}.

\subsection{Step 3: existence of a flat unit vector field} We now prove that the unit of the
 product defined above is the vector field $e=\f{\partial}{\partial u^{n-1}}$, that is
\[c^i_{jk}e^k=\delta^i_j,\qquad\forall i,j=1,\dots,n.\]
For $i\ne 1$ this follows from the results for regular quasihomogeneous pencil \cite{du97}.
 For $i=1$ we have
\[c^1_{jk}e^k=c^1_{j,n-1},\qquad\forall j=1,\dots,n.\]
This means that we have to prove the identities
\begin{eqnarray*}
c^1_{1,n-1}&=&1,\\
c^1_{j,n-1}&=& 0,\qquad\forall j=2,\dots,n.
\end{eqnarray*}
We observe that the functions $c^1_{j,n-1}$ are homogeneous polynomials of the $p$-variables of degree $2(1-j)$. Thus for $j\ne 1$ they vanish. For $j=1$ we have
\[c^1_{1,n-1}\stackrel{\eqref{eq:cons3}}{=}c^n_{n,n-1}=c^n_{nk}e^k=\delta^n_n,\]
where the last equality follows from the fact that $c^i_{jk}e^k=\delta^i_j$ for $i\neq 1$.
It is immediate to check that $\nabla^{\eta} e=0$. Indeed, since $u^n$ is flat, the
 passage from the coordinates $(u^1,\dots,u^n)$ to the flat basic invariants does not
 affect the form of $e$ that remains constant in the new coordinates.

\subsection{Step 4: Identification of the metric $\eta$ with the invariant metric.}
We need a preliminary lemma.

\begin{lemma}
For all $i,j,k=1,\dots,n$
\begin{equation}
c^i_{jk}=c^{n+1-k}_{n+1-i,j}=c^{n+1-j}_{n+1-i,k}.\label{eq:simc}
\end{equation}
\end{lemma}
\begin{proof}
The case $i=1$, and $(j,k)=(n,n)$ is trivial. If $i=1$ and $(j,k)\neq (n,n)$, then
\[
c^1_{jk}\stackrel{\eqref{eq:cons3}}{=}c^{n+1-k}_{nj},
\]
which coincides with the first of the \eqref{eq:simc}. The second one holds true because of the symmetry of the lower indices of the $c^i_{jk}$s, Formula \eqref{eq:commc}. If $i\neq 1$ and $k\neq n$ then
\[
c^i_{jk}\stackrel{\eqref{eq:cons2}}{=}\frac{{\Gamma}^{n+1-j,i}_k}{R^i_i},
\]
and
\[
c^{n+1-k}_{n+1-i,j}\stackrel{\eqref{eq:cons2}}{=}\frac{{\Gamma}^{i,n+1-k}_{j}}{R^{n+1-k}_{n+1-k}}\stackrel{\eqref{eq:id3}}{=}\frac{{\Gamma}^{n+1-k,i}_{j}}{R^{i}_{i}}\stackrel{\eqref{eq:id1}}{=}\frac{{\Gamma}^{n+1-j,i}_{k}}{R^{i}_{i}}\stackrel{\eqref{eq:cons2}}{=}c^i_{jk}.
\]
On the other hand, if $i\neq 1$, $k=n$ and $j\neq n$
\[
c^{n+1-k}_{n+1-i,j}=c^1_{n+1-i,j}\stackrel{\eqref{eq:cons3}}{=}c^{n+1-j}_{n,n+1-i}\stackrel{\eqref{eq:cons2}}{=}\frac{{\Gamma}^{i,n+1-j}_{n}}{R^{n+1-j}_{n+1-j}}\stackrel{\eqref{eq:id3}}{=}\frac{{\Gamma}^{n+1-j,i}_{n}}{R^{i}_{i}}\stackrel{\eqref{eq:cons2}}{=}c^i_{jn}.
\]
Finally if $i\neq 1$ and $(j,k)=(n,n)$, then the three terms of the identity are zero, see \eqref{eq:ineq}.
\end{proof}
We have now all the ingredients to prove that
\begin{equation}
\eta_{is}c^s_{jk}=\eta_{js}c^s_{ik}.\label{eq:inveta1}
\end{equation}
This follows at once from \eqref{eq:simc} and from $\eta_{ij}=\delta_{i,n+1-j}$. In fact
\[
\eta_{is}c^s_{jk}=c^{n+1-i}_{jk}\stackrel{}{=}c^{n+1-j}_{ik}=\eta_{js}c^s_{ik}.
\]

\subsection{Step 5: identification of the cometric $g$ with the intersection form.} We will now prove the identity
\begin{equation}
c^i_{jk}E^k=g^{il}\eta_{lj};\label{eq:affinor}
\end{equation}
which amounts to say that the operator of multiplication by the Euler vector field $E$, defined via the \eqref{eq:cons3},  \eqref{eq:cons4} is the affinor defined composing (the covariant metric) $\eta$ with (the contravariant metric) $g$.
To prove \eqref{eq:affinor}, we write $E=E^i\partial_i$ and first we observe that \eqref{eq:reg} entails
\begin{equation}\label{eq:affinor1}
R^i_j=({\nabla^\eta}_{j}E^i-{\nabla^g}_{j}E^i)=-{\Gamma}^{i}_{jl}E^l,
\end{equation}
which, for $i\neq 1$, yields
\[c^i_{jl}E^l\stackrel{\eqref{eq:cons1}}{=}\f{1}{R^i_i}\eta_{jl}{\Gamma}^{li}_kE^k\stackrel{\eqref{eq:c33}}{=}-\f{1}{R^i_i}\eta_{jl}g^{ls}{\Gamma}^{i}_{sk}E^k\stackrel{\eqref{eq:affinor1}}{=}\f{1}{R^i_i}\eta_{jl}g^{ls}R^i_s\stackrel{\eqref{eq:c4}}{=}\eta_{jl}g^{li}.\]
On the other hand, the case $i=1$ and $j\ne n$ can be reduced to the previous one. In fact
\[c^1_{jl}E^l\stackrel{\eqref{eq:cons3}}{=}c^{n+1-j}_{nl}E^l=g^{n+1-j,l}\eta_{ln}=g^{n+1-j,1}=g^{1l}\eta_{lj},\]
where the other equalities follow from the case $i\neq 1$ and from the explicit form of $\eta$.
Finally, if $i=1$ and $j=n$:
\[c^1_{nl}E^l\stackrel{\eqref{eq:ineq}}{=}c^{1}_{nn}E^n\stackrel{\text{Remark}\;\ref{rem:hom}}{=}c^1_{nn}d_nu^n\stackrel{\eqref{eq:cons4}}{=}\f{n-1}{u^n}\f{n}{n-1}u^n=n=g^{11}=g^{1l}\eta_{ln}.\]
Note that in the first equality we used the explicit form of the Euler vector field, in the fifth the normalization of $g$ (see Remark \ref{remark2:g11}) and in the last the explicit form of $\eta$. 
\newline
\newline
The identity \eqref{eq:affinor} implies
\begin{equation}\label{eq:intf}
g^{ih}=c^i_{jk}E^k\eta^{jh}=c^h_{jk}E^k\eta^{ji}.
\end{equation}
In other words the cometric $g$ can be identified with the intersection form.
\newline
\newline
We prove now an useful identity that we will use later. 
\begin{lemma}
\begin{equation}
g^{is}c^l_{sm}=g^{ls}c^i_{sm},\label{eq:ginv}
\end{equation}
for all $s,m,l=1,\dots,n$.
\end{lemma}
\begin{proof}
If $m\ne n$ and $l\ne 1$ (any $i$) 
\[g^{is}c^l_{sm}\stackrel{\eqref{eq:cons2}}{=}g^{is}\f{\Gamma^{n+1-m,l}_s}{R^l_l}\stackrel{\eqref{eq:id3}}{=}g^{is}\f{\Gamma^{l,n+1-m}_s}{R^{n+1-m}_{n+1-m}}\stackrel{\eqref{eq:id4}}{=}g^{ls}\f{\Gamma^{i,n+1-m}_s}{R^{n+1-m}_{n+1-m}}
\stackrel{\eqref{eq:cons2}}{=}g^{ls}c^{n+1-m}_{n+1-i,s}\stackrel{\eqref{eq:simc}}{=}g^{ls}c^i_{ms}.\]
If $m\ne n$, $l=1$ and $i\ne 1$ (note that if $i=1$ the identity is trivially verified)
\[g^{is}c^1_{sm}\stackrel{\eqref{eq:cons3}}{=}g^{is}c^{n+1-m}_{ns}\stackrel{\eqref{eq:cons2}}{=}g^{is}\f{\Gamma^{1,n+1-m}_s}{R^{n+1-m}_{n+1-m}}\stackrel{\eqref{eq:id4}}{=}g^{1s}\f{\Gamma^{i,n+1-m}_s}{R^{n+1-m}_{n+1-m}}\stackrel{\eqref{eq:cons2}}{=}g^{1s}c^{n+1-m}_{n+1-i,s}\stackrel{\eqref{eq:simc}}{=}g^{1s}c^i_{sm}.\]
If $m=n$, $l=1$ and $i=1$ \eqref{eq:ginv} is trivally true.
On the other hand, if $m=n$, $l=1$ and $i\neq 1$ we have
\begin{align*}
(g^{1s}c^i_{sn}-g^{is}c^{1}_{sn})E^n&=(g^{1s}c^i_{sk}-g^{is}c^{1}_{sk})E^k\nonumber\\
&=g^{1s}g^{ir}\eta_{rs}-g^{is}g^{1r}\eta_{rs}\label{eq:gp}\\
&=0,
\end{align*}
and this implies $g^{1s}c^i_{sn}-g^{is}c^{1}_{sn}=0$ since $E^n=d_nu^n$.  
The first equality follows from \eqref{eq:affinor} and from the fact that \eqref{eq:ginv} holds true if $m\ne n$, $l=1$ and $i\ne 1$, see the previous computation. On the other hand, the last equality is obtained trading $r$ with $s$ in (for example) the second summand. Finally, since $i$ and $l$ appear symmetrically in \eqref{eq:ginv}, the case $m=n$, $i=1$ and $i\neq 1$ follows from the previous computation simply exchanging the role of $i$ and $l$.
\end{proof}

\subsection{Step 6: symmetry of $\nabla c$}
In Saito flat coordinates the vanishing of the curvature of the pencil implies
\begin{equation}
\partial_s\Gamma^{jk}_l=\partial_l\Gamma^{jk}_s,\label{eq:simgamma}
\end{equation}
for all $s,j,k,l=1,\dots,n$, where $\Gamma^{ij}_k$ denote the contravariant Christoffel symbols of the metric $g$, 
see \cite{du97}. This observation entails that
\begin{proposition}
\begin{equation}\label{symnablac}
\partial_sc^{k}_{jl}=\partial_l c^{k}_{js},\,\forall s,j,k,l=1,\dots,n.
\end{equation}
\end{proposition}
\begin{proof}
If $k\neq 1$, then \eqref{symnablac} follows from the definition of the structure constants. In fact in this case $c^{k}_{jl}=\frac{\eta_{jr}\Gamma^{rk}_l}{R^{k}_k}$, where $\frac{\eta_{jr}}{R^k_k}$ are constants. 
If $k=1$ and $(j,l)=(n,n)$, the right-hand side of \eqref{symnablac} is zero unless $s=n$ when this identity is trivially true. If $s\neq n$, then also the left-hand side of \eqref{symnablac} is zero since $n+s>n$.
Finally, if $k=1$ and $(j,n)\neq (n,n)$, then
\[
\partial _s c^1_{jl}=\partial_s c^1_{lj}=\partial_s c^{n+1-j}_{nl}=\partial_l c^{n+1-j}_{ns}
=\partial_lc^1_{sj}=\partial_l c^1_{js}.
\]
\end{proof}

\subsection{Interlude: structure constants of the product and Christoffel symbols} Let $d_F=3-d=2+\frac{2}{n-1}$ and let 
\begin{equation}
c^{ij}_k:=\eta^{is}c_{sk}^j\label{eq:calti}
\end{equation}
for all $i,j,k$, where the $c^j_{sk}$s were defined in \eqref{eq:cons1}, \eqref{eq:cons3} and \eqref{eq:cons4}. Let
\begin{equation}
b^{ij}_k:=\big(1+d_j-\frac{d_F}{2}\big)c^{ij}_k,\qquad\forall i,j,k=1,\dots,n.\label{eq:bs}
\end{equation}
\begin{remark}
Note that for all $j=1,\dots,n$,
\[
1+d_j-\frac{d_F}{2}=\frac{j-1}{n-1}.
\]
\end{remark}
We will prove that
\begin{theorem}\label{thm:LC}
The $b^{ij}_k$s defined in \eqref{eq:bs} satisfy the following equations
\begin{align}
\partial_k g^{ij}&= b^{ij}_k+b^{ji}_k\label{eq:LC1}\\
g^{is}b^{jk}_s &= g^{js}b^{ik}_s,\label{eq:LC2}
\end{align}
for all $i,j,k=1,\dots,n$.
\end{theorem}
To prove this statement we need a couple of preliminary results which we enclose in the following lemmata.

\begin{lemma} Let $c$ the $(1,2)$-tensor field defined by \eqref{eq:cons1}, \eqref{eq:cons3} and \eqref{eq:cons4}. Then
\begin{equation}
\mathcal L_E c=c.\label{eq:chomo}
\end{equation}
\end{lemma}
\begin{proof}
If $c=c^i_{jk}\partial_i\otimes dt^j\otimes dt^k$, since 
\begin{equation}
\mathcal L_E dt^i=\frac{i}{n-1}\,dt^i,\quad\mathcal L_E \partial_i=-\frac{i}{n-1}\,\partial_i\quad\text{and}\quad \text{deg}\; (c^i_{jk})=\frac{n-1+i-j-k}{n-1},\label{eq:degc1}
\end{equation}
see \eqref{eq:degc} above, one has
\begin{align*}
\mathcal L_E c&=(\mathcal L_E c^i_{jk})\partial_i\otimes dt^j\otimes dt^k+c^i_{jk}(\mathcal L_E\partial_i)\otimes dt^j\otimes dt^k\\
&+c^i_{jk}\partial_i\otimes(\mathcal L_E dt^j)\otimes dt^k+c^i_{jk}\partial_i\otimes dt^j\otimes(\mathcal L_E dt^k)\\
&\stackrel{\eqref{eq:degc1}}{=}c.
\end{align*}
For later use, we observe that from the very last equality, solving for $(\mathcal L_E c^i_{jk})\partial_i\otimes dt^j\otimes dt^k$ one obtains:
\beq\label{eq:p2}
E^m\d_mc^j_{lk}=c^j_{lk}+d_jc^j_{lk}-d_lc^j_{lk}-d_kc^j_{lk},
\eeq
where the $d_j$s were defined in \eqref{eq:degh}.
\end{proof}

Once these preliminary results are settled, one can prove Theorem \ref{thm:LC}.

\begin{proof}
First note that \eqref{eq:affinor} implies

\begin{equation}
g^{hk}=\eta^{ki}c^h_{is}E^s.\label{eq:gEta}
\end{equation}
Then we compute
\beq\label{aux22}
\d_k(g^{ij})=\d_k(\eta^{il}c^j_{lm}E^m)=\eta^{il}(\d_k c^j_{lm})E^m+\eta^{il}c^j_{lm}\d_kE^m
\stackrel{\eqref{symnablac}}{=}\eta^{il}(E^m\d_mc^j_{lk})+d_k\eta^{il}c^j_{lk}.
\eeq
Using \eqref{eq:p2} to substitute  $E^m\d_mc^j_{lk}$ in  \eqref{aux22}, we obtain 
\begin{eqnarray}
\d_k(g^{ij})=
\eta^{il}(c^j_{lk}+d_jc^j_{lk}-d_lc^j_{lk}).\label{eq:pat}
\end{eqnarray}
Since the pencil $(g,\eta)$ is homogeneous and exact, 
\[
\mathcal L_E\eta=(d-2)\eta=(1-d_F)\eta,
\] 
see \eqref{eq:hometa} (here $\eta$ denotes the contravariant metric). On the other hand, since $\eta$ is constant when written in the Saito flat coordinates, working with the covariant metric, one has
\begin{align*}
0=\mathcal L_E \big(\eta(\partial_i,\partial_l)\big)&=(\mathcal L_E\eta)(\partial_i,\partial_l)+\eta(\mathcal L_E\partial_i,\partial_l)+\eta(\partial_i,\mathcal L_E\partial_l)\\
&=(d_F-1)\eta(\partial_i,\partial_l)-\partial_iE^m\eta(\partial_m,\partial_l)-\partial_lE^m\eta(\partial_i,\partial_m)\\
&=(d_F-1)\eta^{il}-d_i\eta^{il}-d_l\eta^{il},
\end{align*}
which entails
\[
-\eta^{il}d_l=\eta^{il}(-d_F+1+d_i).
\]
Inserting this identity in \eqref{eq:pat}, one gets
\begin{eqnarray*}
\d_k(g^{ij})=\eta^{il}(2+d_i+d_j-d_F)c^j_{lk}.
\end{eqnarray*}
This should be compared with
\[b^{ij}_k+b^{ji}_k=\left(1+d_j-\f{d_F}{2}\right)c^{ij}_k+\left(1+d_i-\f{d_F}{2}\right)c^{ji}_k.\]
To this end, first one observes that the invariance of the metric $\eta$ w.r.t. the product implies
\begin{equation}\label{eq:simmca}
c^{mh}_k=c^{hm}_k,\,\forall h,m,k=1,\dots,n.
\end{equation}
In fact
\[c^{mh}_k=\eta^{hi}\eta^{mj}\eta_{il}c^l_{jk}=\eta^{hi}\eta^{mj}\eta_{jl}c^l_{ik}=c^{hm}_k.\]
From this one concludes that
\[b^{ij}_k+b^{ji}_k=\left(2+d_i+d_j-d_F\right)c^{ij}_k.\]
To prove \eqref{eq:LC2} we use \eqref{eq:calti}, \eqref{eq:bs}, \eqref{eq:gEta} and  we compute
\begin{eqnarray*}
g^{is}b^{jk}_s&\stackrel{\eqref{eq:bs}}{=}&\eta^{im}c^s_{mh}E^h\left(1+d_k-\f{d_F}{2}\right)\eta^{jl}c^k_{ls}\\
&\stackrel{\eqref{eq:affinor}}{=}&\eta^{im}g^{sh}\eta_{hm}\left(1+d_k-\f{d_F}{2}\right)\eta^{jl}c^k_{ls}\\
&\stackrel{\eqref{eq:ginv}}{=}&\eta^{im}g^{sk}\eta_{hm}\left(1+d_k-\f{d_F}{2}\right)\eta^{jl}c^h_{ls}\\
&\stackrel{\eqref{eq:inveta1}}{=}&\eta^{im}g^{sk}\eta_{hl}\left(1+d_k-\f{d_F}{2}\right)\eta^{jl}c^h_{ms}\\
&\stackrel{\eqref{eq:ginv}}{=}&\eta^{im}g^{sh}\eta_{hl}\left(1+d_k-\f{d_F}{2}\right)\eta^{jl}c^k_{ms}\\
&\stackrel{\eqref{eq:affinor}}{=}&\eta^{jl}c^s_{lh}E^h\left(1+d_k-\f{d_F}{2}\right)\eta^{im}c^k_{ms}\stackrel{\eqref{eq:bs}}{=}g^{js}b^{ik}_s.
\end{eqnarray*}
\end{proof}

Theorem \ref{thm:LC} implies that 

\begin{proposition}
The $b^{ij}_k$s defined in \eqref{eq:bs} are the contravariant Christoffel symbols of the metric $g$ in the Saito flat coordinates, i.e. 
\begin{equation}
b^{ij}_k={\Gamma}^{ij}_k,\,\forall i,j,k=1,\dots,n.
\end{equation}
\end{proposition}

To conclude the proof of Theorem \ref{thm:maint} we are left to show that the product defined by the $c^i_{jk}$s is associative. 

\subsection{Step 7: associativity of the product.}
We start noticing that since $(g,\eta)$ is a flat pencil, expressing the conditions of zero-curvature for the Levi-Civita connection defined by $g_\lambda:=g-\lambda\eta$ in the Saito flat coordinates, one obtains the following set of equations 
\begin{eqnarray}
\label{C1}
&&\d_sb^{jk}_l-\d_lb^{jk}_s=0,\\
\label{C2}
&&b^{ij}_sb^{sk}_l-b^{ik}_sb^{sj}_l=0.
\end{eqnarray}
The first set of conditions \eqref{C1} does not provide additional information since it follow from the symmetry (in the lower indices) of $\nabla^{\eta} c$. Indeed
\begin{equation}
\label{c1}
\left(1+d_k-\f{d_F}{2}\right)\left(\d_sc^{jk}_l-\d_lc^{jk}_s\right)=R^k_k\eta^{jh}\left(\d_sc^{k}_{hl}-\d_lc^{k}_{hs}\right)=0.
\end{equation}
Let us consider the second set of conditions \eqref{C2}. First we note that using the \eqref{eq:bs} and recalling that $R^k_k=(1+d_k-\frac{d_F}{2})$ for all $k$, these conditions can be rewritten as follows
\begin{equation}
\label{c2}
R^k_kR^j_j(c^{ij}_sc^{sk}_l-c^{ik}_sc^{sj}_l)\stackrel{\eqref{eq:simmca}}{=}R^k_kR^j_j(c^{ji}_sc^{ks}_l-c^{ki}_sc^{js}_l)=R^k_kR^j_j\eta^{jh}\eta^{km}(c^i_{hs}c^s_{ml}-c^i_{ms}c^s_{hl})=0.
\end{equation}
The quadratic conditions \eqref{c2}
 entail the associativity of the product defined by the $c^i_{jk}$s, that is 
\[c^i_{hs}c^s_{ml}=c^i_{ms}c^s_{hl},\]
but when one of the index $m$, $h$ is equal to $n$ (of course, if both indices are equal to $n$ the statement is trivially true).

For this reason, to conclude the proof we are left to show that
\begin{equation}\label{eq:final}
c^i_{nl}c^l_{km}=c^i_{kl}c^l_{nm},
\end{equation}
for all possible values of $i,k,m$. It is worth noticing that if $k=n$ the previous identity is trivially satisfied.
We start checking that
\beq\label{c4aux}
c^i_{nl}c^l_{km}-c^i_{kl}c^l_{nm}=0, \quad (m,k,i)\neq (n,n,1).
\eeq
First recall that, since $b^{ij}_k=\Gamma^{ij}_k$, we have  $c^i_{jk}=\frac{b^{n+1-j,i}_k}{R^i_i}=\frac{b^{n+1-k,i}_j}{R^i_i}$ by \eqref{eq:cons2}. By a direct computation 
\begin{eqnarray*}
&&c^i_{nl}c^l_{km}-c^i_{kl}c^l_{nm}=c^i_{n1}c^1_{km}-c^i_{k1}c^1_{nm}+\sum_{l\ne 1}\left(c^i_{nl}c^l_{km}-c^i_{kl}c^l_{nm}\right)\\
&&\stackrel{\eqref{eq:cons3},\;\eqref{eq:cons2}}{=}c^{i}_{n1}c^{n+1-m}_{nk}-c^{i}_{k1}c^{n+1-m}_{nn}+\sum_{l\ne 1}\left(\f{b^{1i}_l}{R^i_i}\f{b^{n+1-m,l}_k}{R^l_l}-\f{b^{n+1-k,i}_l}{R^i_i}\f{b^{n+1-m,l}_n}{R^l_l}\right)\nonumber\\
&&\stackrel{\eqref{eq:cons2},\;\eqref{eq:id3}}{=}\f{b^{1i}_1}{R^{i}_{i}}
\f{b^{1,n+1-m}_k}{R^{n+1-m}_{n+1-m}}-\f{b^{ni}_k}{R^{i}_{i}}
\f{b^{1,n+1-m}_n}{R^{n+1-m}_{n+1-m}}
+\sum_{l\ne 1}\left(\f{b^{1i}_l}{R^i_i}\f{b^{l,n+1-m}_k}{R^{n+1-m}_{n+1-m}}-\f{b^{n+1-k,i}_l}{R^i_i}\f{b^{l,n+1-m}_n}{R^{n+1-m}_{n+1-m}}\right)\nonumber\\
&&=\f{b^{1i}_1}{R^{i}_{i}}
\f{b^{1,n+1-m}_k}{R^{n+1-m}_{n+1-m}}-\f{b^{ni}_k}{R^{i}_{i}}
\f{b^{1,n+1-m}_n}{R^{n+1-m}_{n+1-m}}
+\sum_{l\ne 1}\left(\f{b^{1i}_l}{R^i_i}\f{b^{l,n+1-m}_k}{R^{n+1-m}_{n+1-m}}-\f{b^{n+1-l,i}_k}{R^i_i}\f{b^{1,n+1-m}_{n+1-l}}{R^{n+1-m}_{n+1-m}}\right)\\
&&=\f{b^{1i}_1}{R^{i}_{i}}
\f{b^{1,n+1-m}_k}{R^{n+1-m}_{n+1-m}}-\f{b^{ni}_k}{R^{i}_{i}}
\f{b^{1,n+1-m}_n}{R^{n+1-m}_{n+1-m}}
+\sum_{l\ne 1}\f{b^{1i}_l}{R^i_i}\f{b^{l,n+1-m}_k}{R^{n+1-m}_{n+1-m}}-\sum_{l\ne n}\f{b^{li}_k}{R^i_i}\f{b^{1,n+1-m}_{l}}{R^{n+1-m}_{n+1-m}}\\
&&=\f{b^{1i}_l}{R^i_i}\f{b^{l,n+1-m}_k}{R^{n+1-m}_{n+1-m}}-\f{b^{li}_k}{R^i_i}\f{b^{1,n+1-m}_{l}}{R^{n+1-m}_{n+1-m}}=0.\nonumber
\end{eqnarray*} 

\begin{remark}
In the previous computation, the fourth line follows from the third one, applying \eqref{eq:id1} to both $b^{n+1-k,i}_l$ and $b^{l,n+1-m}_n$. In the fifth line, the second summation stems after declaring $s=n+1-l$ (and then $s=l$) in the second summand of the summation of the fourth line.
\end{remark}

If $(m,k)\ne (n,n)$ and $i=1$, \eqref{eq:final} becomes
\beq\label{c2aux}
c^1_{nl}c^l_{km}=c^1_{kl}c^l_{nm}.
\eeq
By \eqref{c2}, we know that 
\beq 
c^1_{il}c^l_{km}=c^1_{kl}c^l_{im}, \quad i=1,\dots n-1
\eeq 
since we are also assuming $k\neq n$, $m\neq n$. Therefore, \eqref{c2aux} can be
 rewritten in the following equivalent form 
\[(c^1_{il}c^l_{km}-c^1_{kl}c^l_{im})E^i=0,\]
since, for what already proven, the only non-zero contribution in this sum is the one with $i=n$.

Using \eqref{eq:affinor}, one gets
\begin{align}
(c^1_{il}c^l_{km}-c^1_{kl}c^l_{im})E^i &= c^1_{il}E^ic^l_{km}-c^1_{kl}c^l_{im}E^i \nonumber\\
&= g^{1s}\eta_{sl}c^l_{km}-c^1_{kl}g^{ls}\eta_{sm}\nonumber\\
&\stackrel{\eqref{eq:ginv}}{=}g^{1s}\eta_{ml}c^l_{ks}-c^s_{kl}g^{l1}\eta_{sm}\label{eq:eco}\\
&=0,\nonumber
\end{align}
whose last equality is obtained changing $s$ with $l$ in the second summand of \eqref{eq:eco}. Therefore \eqref{c2aux} holds. 
\newline
\newline

As already observed above, if $m=k=n$ (any $i$) \eqref{eq:final} becomes
\[
c^i_{nl}c^l_{nn}-c^i_{nl}c^l_{nn}=0.
\]
We are left to consider the case $m=n$ and $k\neq n$ (any $i$), that is we need to prove 
\beq\label{c7}
c^i_{nl}c^l_{kn}-c^i_{kl}c^l_{nn}=0, \quad k\neq n, \quad \text{ any } i.
\eeq
We first observe that $c^i_{nl}c^l_{ks}-c^i_{kl}c^l_{ns}=0$ for $s=1,\dots, n-1$, for any $i$ since for $i\neq 1$ this is \eqref{c4aux}, while for $i=1$ this is \eqref{c2aux}. Therefore we can rewrite \eqref{c7} in the equivalent form,
\[c^i_{nl}c^l_{ks}E^{s}-c^i_{kl}c^l_{ns}E^{s}=0,\]
which, together with \eqref{eq:affinor}, yields
\[c^i_{nl}g^{ls}\eta_{sk}-c^i_{kl}g^{ls}\eta_{sn}\stackrel{\eqref{eq:ginv}}{=}c^s_{nl}g^{li}\eta_{sk}-c^s_{kl}g^{li}\eta_{sn}\stackrel{\eqref{eq:inveta1}}{=}(c^s_{nl}\eta_{sk}-c^s_{kl}\eta_{sn})g^{li}=0.\]
This concludes the proof of Theorem \ref{thm:maint}.

\endproof

\section{Conclusions and open problems}
In this paper, combining the procedure presented in \cite{ALcomplex}
 for complex reflection groups with a generalization of the classical Dubrovin-Saito procedure, we have obtained
 a non-standard Dubrovin-Frobenius structure on the orbit space of $B_n$, more precisely on the orbit space less the image of coordinates hyperplanes under the quotient map.
 The procedure of \cite{ALcomplex} allowed us to get explicit formulas in the cases $n=2,3,4$  while the generalized Dubrovin-Saito procedure allowed us to prove the existence of this structure for arbitrary $n$. 
 Two main questions are still open:
\newline
\newline
- For $n=2,3,4$ the dual product is defined by
\[*=\f{1}{N}\sum_{H\in \mathcal{H}}\frac{d\alpha_H}{\alpha_H}\otimes\pi_H\]
with  $\sigma_H=0$ for all the mirrors in the Orbit I and $\sigma_H=1$ for all the mirrors in the Orbit II. Is it true for arbitrary $n$?
\newline
\newline
- For $n=2,3,4$ the Dubrovin-Frobenius prepotentials 
\begin{eqnarray*}
F_{B_2}&=& \f{1}{2}u_1^2u_2\pm\f{1}{2}u_2^2\left(\ln{u_2}-\f{3}{2}\right).\\
F_{B_3}&=&\f{1}{6}u_2^3+u_1u_2u_3+\f{1}{12}u_1^3u_3-\f{3}{2}u_3^2+u_3^2\ln{u_3},\\
F_{B_4}&=&\f{1}{108}u_1^4u_4+\f{1}{6}u_1^2u_2u_4-\f{1}{72}u_2^4+u_1u_3u_4+\f{1}{2}u_2^2u_4+\f{1}{2}u_2u_3^2-\f{9}{4}u_4^2+\f{3}{2}u_4^2\ln{u_4},
\end{eqnarray*}
coincide 
 with the solutions of WDVV equations associated with constrained KP equation (see \cite{DaZ}), in particular the case $n=2$ is related to the defocusing NLS equation.
 Is it true for arbitrary $n$?
\newline
\newline
In both cases we expect that the answer is positive.

\end{document}